\documentclass[reqno]{amsart}
\usepackage[T1]{fontenc}
\usepackage{amssymb, amsmath, color, textcomp, url,tikz}
\usetikzlibrary{matrix,arrows}
\usetikzlibrary{fit}
\usetikzlibrary{positioning}
\usepackage[labelformat=empty]{caption}

\usepackage{mdwlist}

\RequirePackage{ifpdf}
\ifpdf
   \usepackage[pdftex]{hyperref}
\else
   \usepackage[hypertex]{hyperref}
\fi

\usepackage{enumerate,xspace}

\theoremstyle{plain}
\newtheorem{theorem}{Theorem}[section]
\newtheorem{cor}[theorem]{Corollary}
\newtheorem{prop}[theorem]{Proposition}
\newtheorem{lemma}[theorem]{Lemma}
\newtheorem*{teoA}{Theorem A}
\newtheorem*{teoB}{Theorem B}

\newtheorem*{claim}{Claim}
\newenvironment{claimproof}{\noindent\textit{Proof of
		Claim.}}{\hfill\qedsymbol \tiny{ Claim}
	\medskip}

\theoremstyle{definition}
\newtheorem{remark}[theorem]{Remark}
\newtheorem{fact}[theorem]{Fact}
\newtheorem{definition}[theorem]{Definition}
\newtheorem*{example}{Example}
\newtheorem*{question}{Question}
\newtheorem*{notation}{Notation}

\newcommand{\nc}{\newcommand}

\nc{\Z}{\mathbb{Z}}
\nc{\Q}{\mathbb{Q}}
\nc{\N}{\mathbb{N}}
\nc{\C}{\mathbb{C}}
\nc{\R}{\mathbb{R}}
\nc{\Ra}{\mathbb{R}_\textrm{alg}}

\nc{\Ima}{\mathrm{Im}}
\nc{\restr}[1]{\!\!\upharpoonright_{#1}}
\nc\LL{\mathcal L}
\nc\1{\operatorname{1}}

\nc{\inte}{\operatorname{int}}

\nc\scl{\operatorname{scl}}
\nc\cl{\operatorname{cl}}
\nc{\dcl}{\operatorname{dcl}}
\nc{\dclq}{\operatorname{acl^\text{eq}}}

\nc{\acl}{\operatorname{acl}}
\nc{\tr}{\operatorname{tr.deg}}
\nc{\ldim}{\operatorname{ldim}}
\nc\inv{ ^{-1}}

\nc{\tp}{\operatorname{tp}}
\nc{\cf}{\text{cf. }}
\nc{\eg}{\text{e.g. }}

\def\Ind#1#2{#1\setbox0=\hbox{$#1x$}\kern\wd0\hbox to
  0pt{\hss$#1\mid$\hss} \lower.9\ht0\hbox to
  0pt{\hss$#1\smile$\hss}\kern\wd0}
\def\Notind#1#2{#1\setbox0=\hbox{$#1x$}\kern\wd0\hbox to
  0pt{\mathchardef\nn="0236\hss$#1\nn$\kern1.4\wd0\hss}\hbox to
  0pt{\hss$#1\mid$\hss}\lower.9\ht0 \hbox to
  0pt{\hss$#1\smile$\hss}\kern\wd0}
\def\ind{\mathop{\mathpalette\Ind{}}}

\def\indip{\mathop{\ \ \hbox to 0pt{\hss$\mid^{\hbox to
0pt{$\scriptstyle P$\hss}}$\hss}
\lower4pt\hbox to 0pt{\hss$\smile$\hss}\ \ }}
\def\nindip{\mathop{\ \ \hbox to 0pt{\hss$\!\not{\mid}^{\hbox to
0pt{$\scriptstyle\, P$\hss}}$\hss}
\lower4pt\hbox to 0pt{\hss$\smile$\hss}\ \ }}

\begin{document}

\title[Dense pairs of topological structures]{Open core and small groups in 
dense pairs of topological structures}
\date{\today}

\author{Elias Baro and Amador Martin-Pizarro}

 \address{Departamento de
  \'Algebra, Geometr\'ia y Topolog\'ia; Facultad de Matem\'aticas; Universidad Complutense de
  Madrid; Plaza de Ciencias, 3; Ciudad Universitaria; E-28040 Madrid;
  Spain}

\address{Abteilung f\"ur Mathematische Logik; Mathematisches
  Institut; Albert-Ludwig-Universit\"at Freiburg; Eckerstra\ss e 1;
  D-79104 Freiburg; Germany}

\email{eliasbaro@pdi.ucm.es}
\email{pizarro@math.uni-freiburg.de}

\thanks{Research partially supported by the program MTM2014-59178-P
  and MTM2017-82105-P} \keywords{Model Theory, Topological Structures,
  Real Closed Fields, Dense Pairs} \subjclass{03C64, 03C45}

\begin{abstract}

  Dense pairs of geometric topological fields have \emph{tame open core}, that is,
  every definable open subset in the pair is already definable in the 
  reduct. We fix a minor gap in the published 
  version of van den Dries's  seminal work on dense pairs of o-minimal groups, 
  and show 
  that every definable unary function in a dense pair of geometric 
  topological fields agrees with a definable function in the reduct, 
  off a small definable subset, that is, a definable set internal to 
  the predicate. 

  For certain dense pairs of geometric topological fields without the 
  independence property, whenever the underlying set of a definable group  is 
  contained in the dense-codense predicate, the group law is locally   
  definable in the reduct as a geometric topological field. If the reduct has 
  elimination of imagi\-na\-ries, we extend this result, up to  
  interdefinability, 
  to all  groups internal to the predicate. 
\end{abstract}

\maketitle

%\section*{English Summary}

\section*{Introduction}

Tame topological structures and expansions by a predicate have
often been considered from a model-theoretical point of view. Both
$p$-adically and real closed fields are naturally endowed with a
definable topology such that the field operations are continuous. The
close interaction between the topological and algebraic pro\-perties
of such structures is crucial to determine their model-theoretic
behaviour. Se\-ve\-ral frameworks have been suggested to treat
simultaneosly archimedean and non-archimedean normed fields: in this
paper, we will follow the topological approach proposed by Mathews
\cite{lM95}, which was later on adopted by Berenstein, Dolich and
Onshuus \cite{BDO11} to study the theory of dense pairs.

Robinson \cite{aR59} showed that the theory of a real closed field $M$
equipped with a dense proper real closed subfield $P(M)$ is
complete. Subsequently Macintyre \cite{aM75} proved the same result
for dense pairs of $p$-adically closed fields. The 
model-theoretical properties of dense pairs of o-minimal 
expansions of ordered abelian groups were thoroughly studied
by  van den Dries \cite{vD98}, who gave an explicit description of
definable unary sets and functions, up to \emph{small sets}. A set is small in 
a pair $(M,P(M))$ if it is contained in the
image of the $P(M)$-points by a semialgebraic function. Though the
theory of the pair is no longer o-minimal, every definable open set in
the pair is actually definable in the reduct of $M$ as an ordered
field, so the theory of the pair has o-minimal open core \cite{vD98, MS99, 
DMS10,
  BH12}. A similar result on the open core of pairs of $p$-adically
and real closed fields has been recently obtained by Point \cite{fP11}
(see as well as work of Tressl \cite{mT18}) as a by-product of the
study of the theory of differentially closed topological fields. In
Section \ref*{S:dense}, we will use a criterium (cf.  Proposition
\ref{P:curvfunc}) of geometric nature in order to provide a new 
proof of the following result:

\begin{teoA}
	Let $(M,E)$ be a dense pair of models of a geometric theory $T$ 
	of topological 
	rings. Assume that every 
	$\LL^P$-definably closed set $A$ is 
	special, that is, \[ \dim(a_1,\ldots, a_n/E)=\dim(a_1,\ldots, 
	a_n/A\cap E) \text{ for all } a_1,\ldots, a_n \text{ in } A,\]  
	where $\dim$ is the dimension as a geometric structure.  
	
	Every  open	$\LL^P$-definable subset of $M^n$ over a special 
	set $A$	is $\LL_A$-definable. Hence, the open core of the 
	theory $T^P$ is tame, for it is definable in the reduct $T$. 
\end{teoA}

 After private communication with van den Dries, we provide in 
 the Appendix 
 a fix for a minor gap in the published version of \cite[Corollary 
3.4]{vD98}, and show, in the broader 
context of geometric topological structures, that every definable unary 
function in the pair agrees off a small subset with a function definable in the 
predicate. Note that the same gap also affects \cite[Theorem 4.9]{BDO11}.

Understanding the nature of groups definable in a specific structure
is a recurrent topic in model theory. In \cite{HP94}, it was shown
that a group definable in a $p$-adically or real closed field $M$ is
locally isomorphic to the connected component of the $M$-points of an
algebraic group defined over $M$. In particular, the group law is
locally given by an algebraic map. Furthermore, if the group is Nash
affine in a real closed field, this local isomorphism extends to a
global isomorphism \cite{HP11}, since the Nash topology on Nash affine
groups is noetherian. A crucial aspect to define the local isomorphism
and the corresponding algebraic group through a group configuration
diagram \cite{eH90} is the strong interaction between the
semialgebraic dimension in $M$ and the transcendence degree,
computed in the field algebraic closure of $M$. This interaction will
be captured in Definition \ref{D:stab_controlled} in an abstract set-up.

A first description of groups definable in dense pairs $(M, P(M))$ of
real closed fields appears in \cite{EGH15}: the group law is
\emph{locally} semialgebraic (off a wider class of certain sets, encompassing 
small sets). However, if the group is small, particularly if
the underlying set is a subset of some cartesian product of $P(M)$,
the above description does not provide any relevant information.  In
this note, we tackle the remaining case and consider groups whose
underlying set is a subset of some cartesian product of $P(M)$,
following the proof of Hrushovski and Pillay \cite{HP94}.  A recent
result of Eleftheriou \cite{pE17} on elimination of imaginaries on
$P(M)$ with the induced structure from the pair (see Section \ref*{S:EI}) 
allows 
us to extend
our description to all small groups, whenever the topological geometric
field has elimination of imaginaries.

The first obstacle we encountered is the lack of a sensible
notion of dimension in the pairs. In an arbitrary pair $(M, P(M))$ of
topological structures, there is a rudimentary notion of dimension,
given by the \emph{small closure}: the union of all small
definable sets. Since we are only interested in small sets, the small
closure is not useful for our purposes. In Section \ref{S:honesty}, we
attach a dimension to definable subsets of $P(M)^k$ in terms of their
honest definition, as introduced in \cite{CS13}, when the topological
geometric structure does not have the independence property.  Using
this dimension, we produce in Section \ref{S:gps} a suitable group
configuration diagram and show the following:

\begin{teoB} Consider a dense pair $(M,E)$ 
	of models of a geometric NIP theory $T$ of 
	topological rings, with $E=P(M)$, in
	the language $\LL^P= \LL\cup \{P\}$.  Assume that $T$ is 
	stably controlled with respect to a 
	strongly 
	minimal $\LL^0$-theory $T^0$ (see Definition 
	\ref{D:stab_controlled}).   If $(G, \ast)$ is
	an $\LL^P$-definable group over a special set $D\subseteq M$ 
	such that
	$G\subseteq E^k$ for some $k$, then the group law is locally 
	$\LL^0$-definable. 
	
	Furthermore, if $T$ has elimination of imaginaries, then the 
	same holds for all small groups $(G, \ast)$. 
\end{teoB}

In case of the reals or the
$p$-adics, we conclude that the group law is locally  an algebraic 
group law.  We do not know whether the group law is globally 
semialgebraic for small groups in dense pairs of $p$-adically or 
real closed fields.

\section{Dense pairs of topological fields}\label{S:dense}

In this section, we will recall the basic results on topological
geometric structures \cite{vD98, BDO11}.

\begin{definition}\label{D:topstr}
  A structure $M$ in the language $\LL$ is a \emph{topological
    structure} if there is a formula $\theta(x,\bar y)$, where $x$ is
  a single variable, such that the collection
  $\{\theta(M, \bar a)\}_{\bar a \in M^s}$ forms a basis of a proper
  topology, that is, a $T_1$ topology with no isolated
  points. 
\end{definition}

\begin{remark}
If $M$ is a topological structure, then so is every model of the theory
of $M$.

 Since the formula $\theta(x;\bar y)$ has the
order property, the theory of a topological structure as above is always unstable
\cite[Proposition 1.2]{aP87}.
\end{remark}

\begin{definition}\label{D:CDP}
A \emph{cell} in a topological structure $(M, \theta)$ is a definable set
$X\subset
M^n$ such that  $\pi(X)$ is open and homeomorphic to $X$, for some projection
$\pi:M^n\to M^m$ with respect to a given choice of coordinates. A 
topological 
structure $(M, \theta)$ has \emph{cell decomposition}
if, whenever $A\subset M$, the domain $X$ of every definable function
$F$ over $A$ can be partitioned into finitely many cells
$X_1,\ldots, X_m$, each definable over $A$, such that $F\restr {X_i}$ is
continuous for every $1\leq i\leq m$. 
\end{definition}

Taking as a definable function the characteristic function of a definable set, 
it follows that every definable set can be partitioned into finitely many 
cells. However, since for the union of cells may itself be a cell, note that 
the above cell decomposition is weaker than the corresponding versions for real 
or $p$-adically closed fields. 

The following definition was introduced in \cite{BDO11} in order to generalize 
the results of \cite{vD98} from the o-minimal context to the p-adically closed one. Notice that the last item in the definition below is slightly 
stronger than the corresponding one in \cite{BDO11}.

\begin{definition}\label{D:geomstr}(\cite[Definition 4.6]{BDO11} \& 
\cite[Definition 6.1]{hS01})

A complete theory $T$  in a language 
$\LL$ is a \emph{geometric theory of topological rings} if all of the following 
conditions hold:
\begin{itemize}

\item The theory $T$ eliminates $\exists^\infty$.
\item Algebraic closure satisfies the
  exchange principle in every model $M$ of $T$, that is, given a subset $A$
  and elements $b$ and $c$ of $M$ such that $c$ is in
  $\acl(A,b)\setminus \acl(A)$, then $b$ is algebraic over
  $A\cup\{c\}$.
\item The theory $T$ has Skolem functions (so definable and algebraic 
closure agree).
	\item Every model $M$ is a topological structure, witnessed by the formula 
$\theta(x,\bar y)$, with cell decomposition. 

\item Every definable unary subset of a model of $T$ is either
    finite or has non-empty interior.
 \item The language $\LL$ contains the ring symbols and $M$ is an integral 
domain, with the natural interpretations, such that the ring operations 
are continuous.

\item We have the following density conditions  (cf. \cite[Definition 
6.1]{hS01}): for every $\theta(M,\bar a )$ and every element $b$ 
in $\theta(M,\bar a )$, there exists an open neighbourhood $V$ of
$\bar a$ (in $M^{|\bar a|})$ and a tuple $\bar c$ such that for all $\bar d$ in 
$V$,
\[ b\in  \theta(M,\bar c) \subseteq \theta(M,\bar d).\]

\noindent Furthermore, the set of $\bar a'$ such 
that \[b \in \theta(M,\bar a') \subseteq \theta(M,\bar a)\] has non-empty 
interior.

\end{itemize}
\end{definition}

\noindent In particular, we will say that two subsets $A$ and $B$ of $M$ are
\emph{independent over a common subset} $C\subseteq A\cap B$, denoted
by $$A\ind_C B,$$ if the geometric dimension
\[ \dim(a_1,\ldots,a_n|C)=\dim(a_1,\ldots,a_n|B)\]
\noindent for any $a_1,\ldots,a_n$ in $A$. It follows from \cite[Theorem
9.5]{lM95} that the above dimension coincides with the topological dimension, which is
definable \cite[Corollary]{lM95}.

\begin{example}
  Real and $p$-adically closed fields, as well as real closed rings (in the 
  sense of Cherlin and Dickmann) are geometric topological rings 
  in the ring language  \cite{lM95, HP94}.
\end{example}

Recall that a \emph{dense pair} of models of a geometric theory $T$ 
topological rings is a model  $M$ of $T$
together with a proper elementary substructure $E\preceq M$ which is dense
in $M$ with respect to the definable topology.

\begin{notation} Given a language $\LL$ and a geometric $\LL$-theory $T$ of 
topological rings, we denote from now on by $T^P$ the theory of dense pairs
$(M, E)$ of models of $T$ in the 
language $\LL^P=\LL\cup\{P\}$, where
$P$ is a unary symbol whose interpretation is the proper submodel $E$.
Henceforth, the symbol $P$ will be used to refer to the theory $T^P$:
in particular, the definable closure in the language $\LL^P$ will be
denoted by $\dcl^P$, whereas $\dcl$ denotes the definable closure in
$\LL$. Given a subset $A\subseteq M$ of parameters, by
$\LL_A$-definable, resp. $\LL^P_A$-definable, we mean $\LL$-definable,
resp.  $\LL^P$-definable, over $A$. 
\end{notation}

Let us recall the description of definable sets in dense pairs given by van den Dries via the notion of a small set \cite{vD98}.
There is a natural notion of smallness and largeness for arbitrary
pairs, as introduced by Casanovas and Ziegler \cite{CZ01} in terms of
saturated extensions. Their definition agrees with van den Dries's
original definition  when the underlying theory is strongly
minimal (cf. \cite[Section 2]{vDG06}). We will use the latter one, which 
resonates with the notion of internality to the predicate from 
geometric 
stability theory.

\begin{definition}\label{D:small}
  A set $X\subseteq M^n$ is \emph{small} in the pair $(M,E)$ over $A$ if there is an
  $\LL_A$-definable function $h:M^m\to M^n$ such that
  $X\subseteq h(E^m)$.

   \end{definition}

In order to track the parameters necessary to describe definable sets in 
	such dense pairs, we introduce the following definition.

\begin{definition}\label{D:special}
	A subset $A$ of $M$ is \emph{special} if $A\ind_{E\cap A} E$.
\end{definition}

The following trivial observation will be used all throughout this work.

\begin{remark}\label{R:special}
	Subsets of $E$ are special. Whenever $B$ is a subset of $E$ and $A$ is
	special, so is $\dcl(A\cup B)$, with $E\cap\dcl(A\cup B)=\dcl(B, E\cap A)$.
\end{remark}

Though neither van den Dries nor Berenstein-Dolich-Onshuus refer to special sets when
studying the
theory $T^P$ in \cite{vD98, BDO11}, the next results follow immediately by 
appropriately
choosing the back-and-forth system:

\begin{fact}(see \cite{vD98, BDO11})\label{F:Dries} Let $T$ be a geometric 
theory of topological rings.
\begin{enumerate}
\item\label{F:complete}  The theory $T^P$ of dense pairs is 
complete.
\item\label{F:special}  The type of a special set is uniquely 
determined by its
quantifier-free $\LL^P$-type.
\item\label{F:induced} Any  $\LL^P$-definable subset $Y$ of $E^n$
 over a special set $A$ is the trace of an
  $\LL$-definable subset $Z\subseteq M^n$ over $A$:
$$ Z\cap E^n= Y.$$
\item\label{F:dcl_special} The $\LL^P$-definable closure of a special
  set $A$ agrees with its $\LL$-definable closure, and
  $E\cap \dcl^P(A)=\dcl(E\cap A)$.
\item\label{F:fcts} Every $\LL^P$-definable function $f:M\to M$ over a
  special set $A$ agrees with an $\LL_A$-definable function off 
  some  small set (cf. Appendix \ref{A:Dries}).
\item\label{F:open_1dim} For every $\LL^P$-definable unary set
  $X\subseteq M$ over a special set $A$, there are pairwise disjoint $\LL_A$-definable open
  sets $U$, $V$
  and $W$ of $M$  such that:

  \begin{itemize}
  \item $M\setminus (U\cup V\cup W)$ is finite.
  \item $U\subset X$ and $X\cap V=\emptyset$.
  \item $X\cap W$ is both dense and codense in $W$.
  \end{itemize}

\noindent Thus, if $X\subseteq M$ is small and $\LL^P$-definable, then the corresponding
open set $U$ is empty.
\end{enumerate}
\end{fact}
\begin{remark}\label{R:small_fte} To prove the above fact in the context of 
geometric theories \cite{BDO11}, it is crucial that a number of results 
in \cite{vD98} also hold in this context. For example, we also have the 
following two relevant properties:

\begin{enumerate}
\item If $(M,E)$ is a saturated model of $T^P$, notice that $M$ is not $\LL^P$-saturated over
the set of parameters $E$, but the rank of $M$ over $E$ (with respect to the pregeometry in $T$) is at least $\kappa$ (see \cite[Lemma 1.5]{vD98}).

\item Every small and $\LL$-definable unary subset of $M$ is finite: The set 
$M\setminus E$ is codense in $M$ (see the Claim in \cite[Thm.\,4.9]{BDO11}) and 
therefore no open $\LL$-definable subset of $M$ can be small \cite[Lemma 
4.1]{vD98}. We finish by applying Fact $\ref{F:Dries}(\ref{F:open_1dim})$.
\end{enumerate}
\end{remark}

If $X\subseteq M$ is open and  $\LL_A^P$-definable, then $X$ is
$\LL_A$-definable, by Fact $\ref{F:Dries}(\ref{F:open_1dim})$, for the corresponding set
$W$ is empty. The above characterization of
$1$-dimensional $\LL^P$-definable open sets as $\LL$-definable sets has been
shown for arbitrary finite cartesian products of the universe  in various settings \cite{vD98,
MS99, DMS10, 	BH12, fP11, mT18}, archimedean and non-archimedean. For the 
sake of 
completeness, we will provide a different
proof of geometric nature (cf. Theorem \ref{T:opendef}), which is an immediate 
consequence of a general criterion (cf. Proposition \ref{P:curvfunc}) to 
determine whether an $\LL^P$-definable correspondence is $\LL$-definable. We 
first need an auxiliary result. 

\medskip
{\bf Henceforth, we work inside a sufficiently saturated dense pair of models 
$(M,E)$ of a geometric theory $T$ of topological rings}.

\begin{lemma}\label{L:lemacurvasima} Given an $\LL^P$-definable unary function $h$
with domain an infinite definable subset $X$ of $M$ over a special subset $A$, there
exist a special superset $D\supseteq A$ and an $\LL_D$-definable curve
	$\alpha:\theta(M, \bar b) \rightarrow M^2$ such that the intersection of
	$\Ima(\alpha)$ with the graph of $h$ is infinite.
\end{lemma}

\begin{proof} If $X$ is not small, Fact $\ref{F:Dries}(\ref{F:fcts})$
	yields the desired result. Thus, suppose that $X$ is small, so
	$X \subseteq f(E^\ell)$, for some $\LL_A$-definable map $f$. By
	choosing a suitable cell decomposition of $M^\ell$, we can assume that
	$X=f(E^\ell\cap C)$, where $C$ is an
	$\LL_A$-definable open cell $C$ of $M^\ell$ and $f\restr{C}$ is
	continuous.

	Consider the $\LL_A$-definable open subset
	$$C_0=\{x\in C: f\text{ is locally constant at } x\}$$
	of $C$. For each $x$ in $C_0$,  the fiber $f^{-1}(f(x))$ is
	$\ell$-dimensional. An easy dimension computation yields that $f(C_0)$ must be finite.

 Since $X$ is infinite, the set $C\setminus C_0$ is infinite as well. If the
	interior of  $C\setminus C_0$ happens to be trivial, a suitable cell 
	decomposition yields an
	$\LL_A$-definable subset $W$ of $C\setminus C_0$ and a projection 
	$\pi:M^\ell\rightarrow M^r$ for some choice of $r$ coordinates,
	with $r<\ell$  such that the restriction of $\pi$ to $W$ produces a 
	homeomorphism between
	$W$ and its open image $\pi(W) \subseteq M^r$. Since $f( (C\setminus C_0)\cap
	E^\ell)$ is infinite, we may assume that $f(W\cap E^\ell)$ is
	infinite as well, and proceed by induction on $r$.

 Therefore, we need only consider the case when $C\setminus C_0$
	has non-empty interior. Renaming its interior as $C$ again, we may assume that
	$f$ is nowhere locally constant on $C$.

 Let now $e$ be in
	$C\cap E^\ell$. By the density condition in Definition \ref{D:geomstr}, 
	there is a basic neighbourhood $\mathcal U$
	defined over some $D_0\subseteq E$ such that $ e \in \mathcal U$. In particular, the
	open set $\mathcal U\cap C$ is defined over the set $D=A\cup D_0\cup\{e\}$, which is
	again special, by  the
	Remark \ref{R:special}.

 We will show that there is a line $J$ defined over $E$ such that  $f(J\cap
\mathcal U \cap E^\ell)$ is infinite. Choose some basic neighbourhood $\mathcal V$ of
$\bar 0$ in $M^\ell$. Again, we may enlarge $D$ to assume that $\mathcal V$ is defined
over $D$. For each $v$ in $\mathcal V$, let $J_v$ be
the line
\[
J_v=\{e+t\cdot v \}_{t \in M}\] and set \[ V=\{ v\in \mathcal V \,|\, f(J_v\cap \mathcal U )
\text{
is finite}  \}.\]

\noindent If $\mathcal V \cap E^\ell \subseteq V$, then choose some $v$ in $\mathcal V
\cap E^\ell$ of dimension $\ell$ over $D$. The ring operations are continuous, so there is a
basic neighbourhood $\theta(M,\bar a_0)$ of $0$ such that for all $t$ in $\theta(M,\bar
a_0)$, we have that $e+t\cdot v$ is
contained in $\mathcal U$. Again, we may assume that $\bar a_0$ lies in $D$, by 
the the density condition in Definition \ref{D:geomstr}.

 The set $f(J_v\cap \mathcal U)$ is finite and contains $f(e)$. Since the 
 function
$f$ is continuous, the preimage
$f\inv(f(e))\cap J_v\cap \mathcal U$ is a non-empty open $\LL$-definable subset of
$J_v\cap \mathcal U$. Thus, so is the $\LL$-definable subset of $M$

\[ \mathcal{T}= \{t \in \theta(M,\bar a_0)\, |\,  f(e+t \cdot v)=f(e)\}.\]

\noindent Pick some $t$ in $\mathcal{T}$ generic over $D\cup\{v\}$. Since
 \[ f(e+t\cdot v)=f(e),\]

\noindent we have that \begin{multline*} \ell+1 =\dim(v,t/D)= \dim(e+t\cdot v, t/D ) = \\
\dim(e+t\cdot v/D) +
\dim(t/D, e+t\cdot v)=
\dim(e+t\cdot v/D) + 1.
\end{multline*}

\noindent As the element $e+t\cdot v$ lies in the $\LL_D$-definable subset
$f\inv(f(e))$ of $M^\ell$, the set $f^{-1}(f(e))$
is $\ell$-dimensional, so it contains an open subset of $M^\ell$, which 
contradicts our
assumption that $f$ was nowhere locally constant on $C$.

Hence, there is some  $v$ in $\mathcal V\cap E^\ell$ which does not lie in $V$,
that is, the set $f(J_v\cap \mathcal U)$ is infinite. Projecting onto a suitable
coordinate, we find a basic open neighbourhood $\theta(M,\bar b)$ defined over $E$ and
an $\LL_E$-definable function $g: \theta(M,\bar b) \to M^\ell$ such that  
$g(\theta(M,\bar
b)\cap E)\subseteq C$ and $f(g(\theta(M,\bar b)\cap E))$ is infinite. As before, we may
assume that  $g$ is $\LL_D$-definable, up to enlarging $D$.

 For every element $x$ in $\theta(M,\bar b)\cap E$, the set $\dcl(D, x)$ is
special and $\LL^P$-definably closed, by the Remark \ref{R:special}. Note that 
it contains $\tilde h(x)=h\circ
f\circ g (x)$. By saturation, there are finitely many $\LL_D$-definable functions
$h_1,\ldots,h_n:\theta(M,\bar b)\rightarrow M$ such that, whenever $x$ lies in
$ \theta(M,\bar b)\cap E$, then $\tilde h(x)=h_i(x)$ for some $1\leq i\leq n$. Without loss of
generality, the function $\tilde h$ coincides with $h_1$ infinitely often on $ \theta(M,\bar
b)\cap E$. Set now

\[\begin{array}{rccl}
\alpha: &  \theta(M,\bar b) & \to &  M^2\\[1mm]
& t & \mapsto & (f\circ g(t),h_1(t))

\end{array}  \]

\noindent By construction, the graph of the curve $\alpha$ has an infinite 
intersection with the image of $h$.
\end{proof}

\begin{remark}\label{R:rings_lemma}
In the proof of the previous lemma, it was only used that $\LL$ contains the 
language of rings in order to define the line $J$. We do not 
know whether there is a more general set-up in which Lemma 
\ref{L:lemacurvasima} holds.
\end{remark}

\begin{definition}\label{D:tame_Pairs}
A geometric theory $T$ of topological rings is \emph{tame for pairs} if every 
$\LL^P$-definably closed subset in a dense pair $(M,E)$ of models of $T$ is special. 
\end{definition}

\begin{example}
Real and $p$-adically closed fields are tame for pairs (it will be  shown in Lemma 
\ref{Lem:StablyTame}). 
\end{example}

\begin{question}
Is the theory of real closed rings tame for pairs? 
\end{question}

\medskip
{\bf For the rest of this section, we will assume that the geometric theory $T$ of 
topological rings is tame for pairs.}.

\begin{definition}
A correspondence $R\subseteq M^n\times M$ is  a \emph{multi-function} if the collection
$R(\bar x)=\{y \in M\,|\, (\bar x, y) \in R\}$ is finite. We will use the 
notation
$R:M^n\dashrightarrow M$ to denote that $R$ is a multi-function.
 \end{definition}

\begin{prop}\label{P:curvfunc}
	Let $R:M^n\dashrightarrow M$ be an $\LL^P$-definable
	multi-function over a special
	subset  $A$.   The multi-function $R$ is $\LL_A$-definable if
	and only if the multi-function
	$R\circ   \alpha$ is  $\LL_D$-definable, whenever $D\supseteq A$ is special
	and $\alpha$ is an
	$\LL_D$-definable curve with domain an open subset of $M$.
\end{prop}

\begin{proof}

	One implication is trivial, so we need only show that $R$ is
$\LL_A$-definable, assuming it satisfies the right-hand condition.
We prove it by induction on $n$. For $n=1$, the curve $\alpha(t)=t$ is
clearly $\LL_A$-definable, and thus so is $R=R\circ\alpha$ by hypothesis.

 Assume thus that $n>1$ and the statement holds for all $k<n$.
Given $b$ in $M$, consider the multi-function

\[ \begin{array}{rccc}
R(\cdot,b):&M^{n-1}&\dashrightarrow & M \\
& x &\mapsto& R(x,b)
\end{array}
\]

\noindent which is clearly $\LL^P_{\dcl^P(A,b)}$-definable. Since it
satisfies the same condition as $R$, we deduce that
$R(\cdot,b)$ is $\LL_{\dcl^P(A,b)}$-definable, by induction on $n$ (Note that 
the assumption that the theory is tame for pairs yields that $\LL^P$-definably 
closed sets are special). Compactness yields an 
$\LL_A$-definable multi-function
\[\widetilde{R}:M^{n-1+\ell}\dashrightarrow M,\]

\noindent and an $\LL_A^P$-definable map
\[g=(g_1,\ldots,g_\ell):M\rightarrow M^\ell\]
\noindent such that $R(x,b)=\widetilde{R}(x,g(b))$ (as sets) for all $(x,b)$ in
$M^n$. Since the union of small sets is again small, there is an
$\LL_A$-definable map $\widetilde{g}:M\rightarrow M^\ell$ such that
$g$ and $\widetilde{g}$ agree off a small $\LL_A^P$-definable subset
of $M$, by  Fact $\ref{F:Dries}(\ref{F:fcts})$. Consider the following $\LL_A$-definable
multi-function:
$$R_0(x,b)= \widetilde{R}(x,\widetilde{g}(b)).$$
As before,  for each $x$ in $M^{n-1}$, the
multi-function
\[ \begin{array}{rccc}
R_x:&M&\dashrightarrow & M \\
& b &\mapsto& R(x,b)
\end{array}
\]
is $\LL_{\dcl^P(A,x)}$-definable . Hence the
set
\[ Y_x:=\{b \in M: R_x(b)\neq R_0(x,b) \textrm{ (as sets)}\}\]
\noindent is $\LL_{\dcl^P(A,x)}$-definable. Since $Y_x$ is small for every $x$ in $M^{n-1}$,
it must be finite, by the Remark \ref{R:small_fte}.  Compactness yields that there is
some natural number $k$ such that $\text{card}(Y_x)\leq k$, for every $x$ in 
$M^{n-1}$. 

By compactness, there is an $\LL^P_A$-definable function 
$\gamma:M^{n-1}\rightarrow M^\ell$ and a formula $\varphi(y,w)$ such that
$Y_x=\varphi(M,\gamma(x))$ for every $x$ in $M^{n-1}$. Consider now the 
$\LL^P_A$-definable set
\[ X=\{a\in M^{n-1}\,:\, Y_a\neq\emptyset\}.\]
 The (finite) Skolem
functions for $\varphi$ produce finitely many $\LL^P_A$-definable functions 
$h_1,\ldots, h_k:M^{n-1}\to M$
such that $Y_x=\{h_1(x),\ldots, h_k(x)\}$ (possibly with repetitions) for every  $x$ in $X$.
If $Y_x$ has cardinality $\ell\leq k$, then note that
$Y_x:=\{h_1(x),\ldots,h_\ell(x)\}$ and $h_j(x)=h_\ell(x)$ for all $j\geq \ell$.

 If $k=0$, then we are done. Otherwise, let $\ell \leq k$ be the least cardinal 
 of a
non-empty fiber $Y_x$, for some $x$ in $X$. We will first show that
$$X_\ell:=\{x\in M^{n-1}: \text{card}(Y_x)=\ell\}$$
is $\LL_A$-definable. Consider the characteristic
function $\1_{X_\ell}$ of $X_\ell$ with domain $M^{n-1}$. Since functions are
multi-functions, we
need only show, by induction, that $1_X\circ \beta$ is $\LL_{D}$-definable,
whenever we take an $\LL_D$-definable curve $\beta:I\rightarrow M$ with domain an open
subset $I$ of $M$, where $D$ is special and $D\supseteq A$.

 By Fact \ref{F:Dries} $(\ref{F:open_1dim})$,  there are $\LL_D$-definable open
sets $U$, $V$ and
$W$ of $M$ such that $M\setminus (U\cup V\cup W)$ is finite, the set $U$ is contained
in
\[Z=\{t \in I\,\mid \1_{X_\ell}\circ\beta\,(t) =1 \},\]
the set
$V$ is disjoint from $Z$ and $Z\cap W$ is dense and codense in $W$.  It 
suffices to show that $W$ is empty. For each $i\leq k$, Fact \ref{F:Dries} 
$(\ref{F:fcts})$ yields that there is an $\LL_D$-definable function
$\gamma_i:I\rightarrow M^{n-1}$ which agrees with $h_i\circ \beta$ off a
small set $S_i$. Note that
the set
$$\{t\in I: R(\beta(t),\gamma_i(t))=R_0(\beta(t),\gamma_i(t))\}$$
is $\LL_D$-definable and contained in the small subset $S_i$, so it must be finite, by the
Remark \ref{R:small_fte}. Without loss of generality, we may  assume that 
$\gamma_i(t)$ belongs to
$Y_{\beta(t)}$. Likewise, the $\LL_D$-definable set 
$$\{t\in I \, \mid\, \#\{\gamma_1(t),\ldots, \gamma_\ell(t)\}<\ell\}$$
is finite, for it is contained in the small set $S_1\cup \cdots \cup S_\ell$. Thus, we may
assume that $\#\{\gamma_1(t),\ldots, \gamma_\ell(t)\}=\ell$,  for all $t$ in
$I$.

Since $W\setminus Z$ is dense in $W$, there is some $1\leq j\leq k$ such that the set
$$Z_j=\{t\in I\,\mid\, h_j(t)\notin \{\gamma_1(t),\ldots, \gamma_\ell(t)\}\}$$
is infinite (note that it could be $j=1$).  Lemma \ref{L:lemacurvasima} yields a special set
$D_1\supseteq D$ and  an $\LL_{D_1}$-definable curve $\alpha=(\alpha_1,\alpha_2)$  such
that the intersection of
$\Ima(\alpha)$ with the graph of $(h_j\circ\beta)\restr{Z_j}$ is infinite. Consider the
$\LL_{D_1}$-definable set
\begin{multline*}
\{ t \in W \,\mid\, t=\alpha_1(u) \text{ for some } u \in \mathrm{Dom}(\alpha) \ \& \
\alpha_2(u)\notin \{\gamma_1(t),\ldots, \gamma_\ell(t)\}\\
\& \ R(\beta \circ \alpha_1(u), \alpha_2(u))\neq 
R_0(\beta\circ\alpha_1(u),\alpha_2(u))
\},
\end{multline*}
which is clearly infinite because it contains the projection on the first coordinate of the
intersection of $\Ima(\alpha)$ with the graph of $(h_j\circ\beta)\restr{Z_j}$. Thus, it must
have non-empty interior, so find an $\LL_{D_1}$-definable open  subset of $W$
contained in $W\setminus Z$, which  contradicts
the fact that $Z$ is dense in $W$. Hence, the set $X_\ell$ is 
$\LL_A$-definable, as desired.

Now, consider the $\LL^P_A$-definable multi-function
$$H_\ell:X_\ell\dashrightarrow M:x\mapsto \{h_1(x),\ldots,h_k(x)\}=\{h_1(x),\ldots,h_\ell(x)\},$$
and let us show that $H_\ell$ is $\LL_A$-definable, by induction on $n$. 
Let $\beta:I\rightarrow X_\ell$ be an $\LL_D$-definable curve, with $D\supseteq 
A$ special. For each $i\leq k$, there is an $\LL_D$-definable function
$\gamma_i:I\rightarrow M^{n-1}$ which agrees with $h_i\circ \beta$ off a
small set $S_i$. As before, for all $t\in I$ and $i=1,\ldots,\ell$ we can assume that  $\gamma_i(t)\in Y_{\beta(t)}$ and $\text{card}\{\gamma_1(t),\ldots, \gamma_{\ell} (t)\}=\ell$. In particular,
$$H_\ell \circ \beta(t)=\{\gamma_1(t),\ldots, \gamma_{\ell}(t)\}$$
is $\LL_D$-definable. Thus,  we deduce that $H_\ell$ is $\LL_A$-definable, as 
required.

By (finite) Skolem functions for the graph of $H_\ell$, there are 
$\LL_A$-definable functions $\tilde{h}_1,\ldots, 
\tilde{h}_\ell:X_\ell\rightarrow M$ such that 
$H_\ell(x)=\{\tilde{h}_1(x),\ldots, \tilde{h}_\ell(x)\}$ for all $x\in X_\ell$. 
In particular, for each $1\leq j\leq \ell$, the multi-function 
$R(x,\tilde{h}_j(x))$ with domain
$X_\ell$ is $\LL_A$-definable: Indeed, given an $\LL_D$-definable curve
$\beta:I \rightarrow X_\ell$ over a special set $D\supseteq A$, the curve
$x\mapsto (\beta(x),\tilde{h}_j(\beta(x)))$ is $\LL_D$-definable. By hypothesis, so is the
multi-function $x\mapsto R(\beta(x),\tilde{h}_j(\beta(x)))$. By induction on 
$n$, we deduce that
$x\mapsto R(x,\tilde{h}_j(x))$ is
$\LL_A$-definable, as desired.

Using the following multi-function $\widehat{R}$ defined as
\[ \begin{array}{ccl}
\widehat{R}:M^{n} &\rightarrow & M \\

(x,y)& \dashrightarrow & \begin{cases}
R(x,\tilde{h}_j(x)), \text{ if } x\in X_\ell  \text{ and }
y=\tilde{h}_j(x) \text{, for some $1\leq j\leq k$,}\\
R_0(x,y) , \text{ otherwise}.
\end{cases}
\end{array}
\]

\noindent iterating the previous argument, we find a
multi-function $\widehat{R}_0:M^n\dashrightarrow M$ which is $\LL$-definable 
over $A$ such that the corresponding 
set
\[ \widehat{Y}_x:=\{b \in M: R(x,b)\neq \widehat{R}_0(x,b)\}\]
\noindent is again $\LL_{\dcl^P(A,x)}$-definable and finite for each $x$ in 
$M^{n-1}$. Moreover, we know that $\text{card}(\widetilde{Y}_x)\leq k$ for 
every $x\in M^{n-1}$. On the other hand, the least cardinality $\tilde{\ell}$ 
of a non-emtpy fiber $\widetilde{Y}_x$ for $x$ in $M^{n-1}$ is strictly greater 
than $\ell$. This process must therefore terminate after finitely many 
iterations, so  the 
multi-function $R$ is $\LL_A$-definable, as required.\end{proof}

We obtain an immediate consequence of the previous result, 
considering characteristic functions of $\LL^P$-definable sets.

\begin{cor}\label{C:charfct}
  An $\LL^P$-definable subset $X$ of $M^n$ over a special set $A$ is
  $\LL_A$-definable if and only if the set
  $\{t\in I: \alpha(t)\in X\}$ is $\LL_{D}$-definable, whenever
  $\alpha:I\rightarrow M^n$ is an $\LL$-definable curve over a special
  set $D\supseteq A$.
\end{cor}

In particular, we conclude that the open core of the theory $T^P$ is tame, for it is
definable in the reduct $T$ (cf. \cite{DMS10, BH12, fP11, mT18}).
\begin{theorem}\label{T:opendef}
Let $(M,E)$ be a dense pair of models of a geometric theory $T$ of topological 
rings such that $T$ is dense for pairs, that is, every $\LL^P$-definably closed set is 
special. Every open
$\LL^P$-definable subset of $M^n$ over a special set $A$
  is $\LL_A$-definable. In particular, the topological closure
  $\text{cl}_M(X)$ of an $\LL_A^P$-definable set $X$ of $M^n$ is
  $\LL_A$-definable.
\end{theorem}
\begin{proof}

  Let $U$ be an open subset of $M^n$ which is $\LL^P$-definable over a
  special set $A$. By Corollary \ref{C:charfct}, it suffices to show
  that, given an $\LL$-definable curve $\alpha$ over
  a special set $D\supseteq A$, the set $\{t\in \mathrm{Dom}(\alpha) \,|\, \alpha(t)\in U\}$ is
  $\LL_{D}$-definable. Without loss of generality, we may assume that
  $\alpha$ is continuous. In particular, the preimage
\[ \alpha^{-1}(U)=\{t\in \mathrm{Dom}(\alpha) \,|\, \alpha(t)\in U\}\]
\noindent is an open $\LL_D^P$-definable subset of $M$, and thus
$\LL_D$-definable, by Fact $\ref{F:Dries}.(\ref{F:open_1dim})$.
\end{proof}

\section{Honesty and Dimension}\label{S:honesty}

By Fact $\ref{F:Dries}.(\ref{F:induced})$, any $\LL^P$-definable
subset $Y$ of $E^n$ is \emph{externally $\LL$-definable}: there is
some $\LL$-definable subset $Z$ in $M^n$, possibly with parameters
from $M$, such that $Y=Z\cap M$. The possible external definitions for
a given subset may be very different, as the following example shows:

\begin{example} In the dense pair $(\R,\Ra)$, consider the
  $\LL^P$-definable set \[Y=\{(x,y)\in \Ra^2:x=y=0\}.\] Set
  $Z=\{(x,y)\in \R^2:x=y=0\},$ which has dimension
  $0$. Clearly $Y=Z\cap \R^2_{\text{alg}}$. However, for
  $Z_1=\{(x,y)\in \R^2:x=e\cdot y\}$, which has dimension $1$, we
  still have that $Y=Z_1\cap \Ra^2$.
\end{example}

In a stable theory, externally definable subsets of any
predicate are actually definable using parameters from the predicate,
since $\varphi$-types are definable.  For a general NIP theory, we can
find external definitions using parameters in some elementary
extension of the pair \cite{CS13}.

\begin{fact}\label{F:BDO}\cite[Corollary 4.10]{BDO11}
If the geometric theory $T$ of topological rings has NIP, then so does the theory 
$T^P$ of dense pairs.
\end{fact}

\medskip
{\bf Henceforth, we work inside a sufficiently saturated dense pair of models 
	$(M,E)$ of a geometric NIP theory $T$ of topological rings}.

\begin{fact}\cite[Theorem 3.13 \& Proposition 3.23]{S15}\label{F:honest}
  Let $M$ be a model of $T^P$ with $E=P(M)$, and
  $\varphi(E,b)\subseteq E^n$ be an externally definable set, with
  $\varphi(x,y)$ an $\LL$-formula and $b$ in $M$. In some
  elementary extension $M\preceq^P N$, there is an \emph{honest definition} 
  $\chi(x)$ of
  $\varphi(E,b)$, that is, an $\LL$-formula $\chi(x)$
  with parameters in $P(N)$ such that
$$\varphi(E,b)\subseteq \chi(P(N))\subseteq \varphi(P(N),b).$$

\noindent We say that the set  $\chi(P(N))$ is an \emph{honest definition} of 
$\varphi(E,b)$.
\end{fact}

\begin{remark}\label{r:proj} Notice that
$\chi(E)=\varphi(E,b)$, so $\chi(x)$ is indeed an external
definition. Moreover,
$$\chi(P(N))\subseteq \psi(P(N)), \text{ for any $\LL_E$-formula
  $\psi(x)$ with $\varphi(E,b)\subseteq \psi(x)$}.$$
\noindent
The honest definition does not depend on the formula
$\varphi(x,y)$ but solely on the set $\varphi(E,b)$. Furthermore, if $Z$ is an 
honest definition of the externally definable subset
$X$ of $ E^{n+m}$ and $\pi: E^{n+m}\to E^n$ is
projection onto $m$ many coordinates, then $\pi(Z)$ is an honest definition of 
$\pi(X)$ 
\end{remark}

We will define the dimension of an externally definable
subset of $E^n$ in terms of the dimension of some honest
definition. For that, we need the following auxiliary lemma. Given
a subset $C\subseteq M$, we denote by a
\emph{$C$-$n$-ball}  the  set
$$B^C(\bar a)=\theta(C,\bar a_1)\times \cdots\times \theta(C,\bar a_n), \text{ 
for some tuple } \bar{a}=(\bar a_1,\ldots,\bar a_n) \text{ in } C.$$

\begin{lemma}\label{L:ball}
  Let $N$ be an $\LL$-elementary extension of a model $M$ of $T$ and
  $X \subseteq N^n$ be an $\LL_N$-definable set containing some $M$-$n$-ball
  $B^M(\bar a)$.  The topological dimension $\dim(X)$ of $X$ equals
  $n$.
\end{lemma}

\begin{proof}

  Since $\dim(X)\leq n=\dim(N^n)$, we need only show $\dim(X)\geq n$, by induction on
  $n$. For $n=1$, if $\dim(X)$ were $0$,
  then $X$ would be finite, which contradicts the fact that, for a basic open 
  set defined over $M$, the set  $\theta(M, \bar a)$ is infinite.

 Assume now that $X$ is an $\LL_D$-definable subset of $N^{n+1}$, for some 
 $D\subseteq N$, and let
  $\pi:N^{n+1}\to N^n$ be
  the projection onto the first $n$ coordinates. Consider the definable set
  $$Y:=\{(b_1,\ldots,b_n)\in \pi(X): \pi^{-1}(b_1,\ldots,b_n)\cap X \text{ is infinite} \}.$$
  Denoting by $\bar a = 
  (\bar a_1,\ldots,\bar a_{n+1})$ and $\bar a'=(\bar a_1,\ldots,\bar a_{n})$, 
  we have that $B^M(\bar a')=\pi(B^M(\bar a))\subseteq Y$. 
  By induction $\dim(Y)=n$, so there is $(b_1,\ldots,b_n)$ in $Y$ such that 
  $\dim(b_1,\ldots,b_n|D)=n$. Since $\pi^{-1}(b_1,\ldots,b_n)\cap X$ is 
  infinite, there is $b_{n+1}$ in $N$ such that 
  $\dim(b_{n+1}|b_1,\ldots,b_n,D)=1$ and $(b_1,\ldots,b_n,b_{n+1})$ belongs to 
  $X$, as  
  required.
\end{proof}

\begin{cor}\label{L:good_dim}
  The  topological dimension of any two honest definitions of an
  $\LL^P$-definable set $Y\subseteq E^n$ coincide.
\end{cor}

\begin{proof}
  Choose two $\LL^P$-elementary extensions $N_1$ and $N_2$ of $M$ such
  that $Z_1=\chi_1(P(N_1),d_1)$ and $Z_2=\chi_2(P(N_2),d_2)$ are
  both honest definitions of $Y=\varphi(E,b)$, with $b$ in $M$ and
  $d_i$ in $P(N_i)$, for $i=1,2$.

 By Remark \ref{r:proj}, we may assume that $\dim(Z_1)=n$.
  There exists a $P(N_1)$-$n$-ball
\[\theta(P(N_1),\bar a_1)\times\cdots\times \theta(P(N_1),\bar a_n)\]

\noindent   contained in
  $\chi_1(P(N_1),d_1)\subseteq \varphi(P(N_1),b)$. The density condition in 
  Definition \ref{D:geomstr} implies that we may assume that the tuples $\bar 
  a_1,\ldots,  \bar a_n$ belong to $P(N_1)$.
  Hence,
  \[ N_1\models \exists \bar z_1  \in P \ldots \exists  \bar z_n \in P
    \ \forall x_1 \in P \ldots \forall x_n \in P  \left(\bigwedge\limits_{i=1}^n \theta(x_i,\bar z_i)
    \Longrightarrow \varphi(\bar x,b) \right). \]

\noindent  Thus, the model $M$ must also satisfy the above formula.
Choose a $P(M)$-$n$-ball contained in
  $\varphi(E,b)\subseteq \chi_2(P(N_2), d_2)$. Since
  $E=P(M)\preceq P(N_2)$, we conclude that $\dim(Z_2)=n$, by Lemma \ref{L:ball}.
\end{proof}

\begin{definition}\label{D:dim_hon}
  The \emph{dimension} $\dim(Y)$ of an $\LL^P$-definable subset $Y$ of
  $E^n$ is the topological dimension of some (or equivalently, any) honest 
  definition of $Y$.

 If $Y$ is $\LL^P$-definable over $A$, a point $b$ in $Y$
  is \emph{generic} over $A$ if its geometric dimension $\dim(b/A)$ equals $\dim(Y)$.
\end{definition}
 For an $\LL^P$-definable subset $Y$ of $E^n$, notice that $Y\times Y$
has dimension $2\dim(Y)$, since $Z\times Z$ is an honest definition of
$Y\times Y$, whenever $Z$ is an honest definition of $Y$. If
$Y\subseteq E^n$ is $\LL$-definable over $E$, then $Y$ is an honest
definition of itself, so its dimension as defined above agrees with
its topological dimension.

\medskip
We now give some equivalent (geometric) characterizations of the
introduced dimension. We first compare it with the weakly o-minimal dimension from 
\cite{MMS}.

\begin{lemma}\label{L:proj_dim_coordinates}

  Let $Y$ be an $\LL^P$-definable subset of $E^n$. Then $\dim(Y)$ is
  the largest integer $k$ such that the subset $\pi(Y)$ of $E^k$ has
  non-empty interior, for some projection $\pi:E^n \rightarrow E^k$
  onto $k$ coordinates.
\end{lemma}

\begin{proof}

  In some $\LL^P$-elementary extension $N$ of $M$, let
  $Z=\chi(P(N),d)$ be an honest definition of $Y=\varphi(E,b)$, with
  $b$ in $M$ and $d$ in $P(N)$. Let $k\leq n$ be the largest integer
  such that for some projection  $\pi:E^n \rightarrow E^k$, the set  $\pi(Y)$ 
  has non-empty interior in $E^k$. By Fact \ref{F:honest}, the set $\pi(Z)$ is
  an honest definition of $\pi(Y)$. The latter contains an $E$-$k$-ball
  $B^E(\bar a)$, for some tuple $\bar a$, which we may assume lies in $E$, by 
  the density condition in \ref{D:geomstr}. Since $E\preceq P(N)$, Lemma
  \ref{L:ball} gives that $k=\dim \pi(Z)\leq \dim(Z)=\dim(Y)$.

 In order to show that $\dim(Y)\leq k$, choose $\dim(Y)=\dim(Z)$ many
coordinates such that, for the corresponding projection $\pi:P(N)^n\rightarrow
P(N)^{\dim Y}$, the set $\pi(Z)$ has non-empty interior in $P(N)^{\dim Y}$. In
  particular, the set $\pi(\varphi(P(N),b))$ has non-empty
  interior. Since $M \preceq^P N$, we deduce that $\pi(\varphi(E,b))$
  has non-empty interior in $E^{\dim(Y)}$, as desired.
\end{proof}

\begin{cor}\label{C:dim_incr}

  Let $Y_1 \subseteq Y_2$ be $\LL^P$-definable subsets of $E^n$. Then
  $\dim(Y_1)\leq \dim(Y_2)$.
\end{cor}

\begin{lemma}\label{L:dim_subsets}

	Given an $\LL^P$-definable subset $Y$ of $E^n$,
$$\dim(Y)=\max\{ \dim(Z_0): Z_0 \subseteq Y \text{ and } Z_0 \text{ is }
\LL\text{-definable over } E\}.$$
\end{lemma}

\begin{proof} By Corollary \ref{C:dim_incr}, the value $\dim(Y)$ is at
  least the maximum of the values on the right-hand side of the
  equation. We need only prove the other inequality. Choose an honest
  definition $Z=\chi(P(N),d)$ of $Y=\varphi(E, b)$ in some
  $\LL^P$-elementary extension $N$ of $M$, with $b$ in $M$ and $d$ in
  $P(N)$.  Since $M\preceq^P N$ and the topological dimension is definable, there is some
  $\tilde d$ in $E$ such  that $\chi(E,\tilde d)\subseteq \varphi(E,b)=Y$ with
  $\dim(\chi(M,\tilde d))=\dim(Y)$, as desired.
\end{proof}

Density of the predicate $P$ implies the existence of 
(non-standard)
elements of the appropriate dimension for any honest definition of an
$\LL^P$-definable subset of $E^n$.

\begin{lemma}\label{L:dim_hon}

	Let $Y\subseteq E^n$ be an
	$\LL^P$-definable subset with honest definition $Z\subseteq P(N)^n$ having
	parameters over $P(N)$, for some $\kappa$-saturated elementary
	extension $M\preceq^P N$. Whenever $Z$ is definable over a subset $B$ of $N$ of size
	strictly less than $\kappa$, there is a generic element $a$ in $Z$, that is, with $\dim(a/B)=\dim(Z)$.
\end{lemma}

\begin{proof}
	By Remark \ref{r:proj}, we may assume that $\dim(Z)$ equals $n$.
	Write $Z=\chi(P(N),d)$, for some tuple $d$ in $P(N)$.

We need only show that the   collection of formulae
	$$\Sigma(x)=\{\chi(x,d)\} \cup \{ x\in P\} \cup \{ \neg
	\psi(x,b):   \dim(\psi(x,b))<n \}_{\begin{subarray}{l} b \in B
		\\ \psi(x,y) \ \LL-\text{formula}   \end{subarray}},$$
	\noindent is finitely satisfiable.  Indeed, given
	$\LL_B$-definable subsets $X_1,\ldots,X_k$ of $N^n$ with
	$\dim(X_i)<n$, the $\LL$-definable set
	\[\chi(N,d)\setminus \bigcup\limits_{i=1}^k X_i\] has
	dimension $n$, so it
	contains some non-empty open subset of $N^n$. By density, it contains
	an element $a$ in $P(N)^n$, as desired.
\end{proof}

\begin{cor}\label{C:dim_min}
If $Y$ is $\LL^P$-definable subset of $E^n$ and $X$ is an $\LL$-definable
subset of $M^n$ such that $Y\subseteq X$, then $\dim(Y)\leq \dim(X)$.
\end{cor}

\begin{proof}

There is some finite subset $B\subseteq M$ such that $Y=\varphi(E)$ and
 $X=\psi(M)$, for some $\LL_B$-formulae $\varphi$ and $\psi$. In a sufficiently
 saturated elementary extension $M\preceq^P N$, choose an honest
 definition $Z\subseteq P(N)^n$ of $Y=\varphi(E)$. By Lemma \ref{L:dim_hon},
 there is a generic point $a$ in $Z$ over $B$.
 As $Z\subseteq \varphi(P(N))\subseteq \psi(N)$, we obtain
  $\dim(Y)=\dim(a/B)\leq \dim(X)$, as desired.
\end{proof}

\medskip
{\bf For the rest of this section, we assume that our geometric NIP theory $T$ 
of topological 
rings is tame for pairs (cf. Definition \ref{D:tame_Pairs})}. 

\medskip
By Theorem \ref{T:opendef}, the topological 
closure $\cl_{M^n}(Y)$ of an $\LL^P$-definable 	subset $Y$ of $E^n$ is 
$\LL_E$-definable. We will now show that the dimension of the $E$-trace of the 
boundary
$(\cl_{M^n}(Y)\setminus Y)\cap E^n$ is strictly less than the
dimension of $Y$.

\begin{prop}\label{P:dim_tame} Let $Y$ be an $\LL^P$-definable subset of $E^n$.
\begin{enumerate}
\item\label{P:dim_topcl} If $\cl_{M^n}(Y)$ denotes the topological
  closure, then $\dim(Y)=\dim(\cl_{M^n}(Y) )$;
\item\label{P:dim_max}$\dim(Y)=\min\{ \dim(X): X\subseteq M^n \text{ is
  }\LL\text{-definable and }  Y\subseteq X \}$;
\item\label{P:dim_bound}
  $\dim((\cl_{M^n}(Y) \setminus Y)\cap E^n)<\dim(Y)$.
\end{enumerate}
\end{prop}

\begin{proof}
  We will prove $(\ref{P:dim_topcl})$, by induction on $n$. By Corollary
  \ref{C:dim_min}, we need only prove that $\dim(Y)\geq \dim( \cl_{M^n}(Y)) $.
  For $n=1$, the result is obvious, since an externally definable subset of
  $E$ of dimension $0$ is finite, and thus so is its topological closure.

Assume first that $ \dim(\cl_{M^n}(Y)) =n$ and write
$Y=X\cap E^n$, for some $\LL$-definable subset $X$ of
  $M^n$. Note that $\dim(X)=n$, for otherwise the set
  $\cl_{M^n}(Y)\setminus X$ has non-empty interior, so it contains
  a basic non-empty open subset $\mathcal U=\theta(M,\bar
  a_1)\times\cdots\times\theta(M,\bar a_n)$, but 
  \[\emptyset \neq \mathcal U\cap Y=\mathcal U\cap X\cap E^n
  \subseteq \mathcal U\cap X =\emptyset.\] 
The density condition in Definition \ref{D:geomstr} implies that 
the set $X$ 
contains a basic non-empty open subset
$$\theta(M,\bar
b_1)\times\cdots\times\theta(M,\bar b_n),$$
\noindent for some tuples $\bar b_1,\ldots, \bar b_n$ in $E$.
Since
\[\theta(E,\bar
b_1)\times\cdots\times\theta(E,\bar b_n)\subseteq X\cap E^n =Y,\]

  \noindent we deduce that
  $\dim(Y)\geq n$, by Corollary \ref{C:dim_incr}.

 In case $\dim (\cl_{M^n}(Y)) <n$, choose some projection
  $\pi:M^n\rightarrow M^k$, with $k=\dim(\cl_{M^n}(Y)) <n$, such that
  $\pi((\cl_{M^n}(Y)))$ has non-empty
  interior in $M^k$. By induction, we have that
  $\dim(\pi (Y))=\dim(\cl_{M^k}(\pi (Y)))$. If
  $\dim(\cl_{M^k}(\pi (Y)))=k$ then
  $k=\dim(\cl_{M^k}(\pi (Y)))=\dim(\pi(Y))$ and thus $\dim(Y)\geq k$,
  by Lemma \ref{L:proj_dim_coordinates}.

  Hence, we need only show that $\dim(\cl_{M^k}(\pi (Y)))$ cannot be
  strictly less than $k$.  Otherwise, the
  set $\pi (\cl_{M^n}(Y))\setminus \cl_{M^k}(\pi (Y))$ has non-empty
  interior, so choose an element $a$ in
  $\pi(\cl_{M^n}(Y))$ and a basic non-empty neighbourhood $\mathcal U$ in $M^k$ 
  such  that
  $a$ belongs to $\mathcal U \subseteq \pi 
  (\cl_{M^n}(Y))\setminus \cl_{M^k}(\pi (Y))$.  
  In particular, there is
  some tuple $\tilde a$  in $\cl_{M^n}(Y)$ such that $\pi(\tilde a)=a$, so the 
  open set 
  $\pi^{-1}(\mathcal U)$ contains the accumulation point $\tilde a$.  Hence the 
  set
  \[ Z= \pi\inv(\mathcal U)\cap Y  \]

  \noindent cannot be empty. However, the set
  $\pi(Z)$ is contained in $\mathcal U\cap \pi(Y)=\emptyset$,
  which is a contradiction.

  In order to prove $(\ref{P:dim_max})$, observe that $\dim(Y)$ is
  less than or equal to the right-hand side, by Corollary
  \ref{C:dim_min}. Let $X$ be an $\LL$-definable subset of $M^n$
  containing $Y$ of least possible dimension.  Notice that
  $X\cap \cl_{M^n}(Y)$ is again $\LL$-definable, by Theorem
  \ref{T:opendef}, and contains $Y$. In particular,
   $$\dim(Y) \leq \dim(X)\leq \dim(X\cap \cl_{M^n}(Y))\leq 
   \dim(\cl_{M^n}(Y))\stackrel{(\ref{P:dim_topcl})}{=}\dim(Y).$$

   We conclude now by showing $(\ref{P:dim_bound})$. Let $X$ be an
   $\LL$-definable subset of $M^n$ of least possible dimension
   containing $Y$.  If $Y=X_1\cap E^n$, for some $\LL$-definable
   subset $X_1$ of $M^n$, notice $X\cap X_1$ is again $\LL$-definable
   and $Y=X\cap X_1\cap E^n$, so
   \[ \dim(X)\stackrel{(\ref{P:dim_max})}{=}\dim(Y)\leq \dim(X\cap 
   X_1)\leq\dim(X).\]
   \noindent We may therefore assume that $Y=X\cap E^n$. We have that
  \begin{multline*} (\cl_{M^n}(Y)\setminus Y)\cap E^n\subseteq  
 (\cl_{M^n}(Y)\cap E^n)
  \setminus Y \subseteq (\cl_{M^n}(X)\cap E^n)\setminus Y   \subseteq \\
  (\cl_{M^n}(X)\cap E^n)\setminus X \subseteq (\cl_{M^n}(X)\setminus
  X)\cap E^n.\end{multline*}

 \noindent Corollaries \ref{C:dim_incr} and \ref{C:dim_min} yield that
  \begin{multline*} \dim((\cl_{M^n}(Y)\setminus Y)\cap E^n)\leq
  \dim((\cl_{M^n}(X)\setminus X)\cap E^n)\leq \\
\dim(\cl_{M^n}(X)\setminus X)< \dim(X)
\stackrel{(\ref{P:dim_max})}{=}\dim(Y),\end{multline*}  as desired.
\end{proof}

\section{Small groups}\label{S:gps}

Every definable group $G$ in a geometric structure $M$ becomes
naturally a topological group with respect to a definable topology on $G$ 
(which need not coincide with the induced topology) \cite[Theorem 
7.1.10]{aM96}.  The study of 
groups in \cite{HP94} rely on the close relation between a real or 
	p-adically closed field and its algebraic closure. For a similar analysis of 
	groups for geometric topological fields, we introduce 
	the following notion.

\begin{definition}\label{D:stab_controlled} A geometric theory $T$ is 
\emph{stably controlled} if 
	there is a strongly minimal theory $T^0$ which eliminates quantifiers
	and imaginaries in a sublanguage $\LL^0\subseteq \LL$ such that:
	\begin{itemize}
		\item The theory $T$ contains $T^0_\forall$, that is, every model $M$
		of $T$ embeds as an $\LL^0$-substructure of some model $N$ of
		$T^0$. In particular, the $\LL^0$-definable closures
		of $M$ in the various models of $T^0$ are definably isomorphic
		(since $T^0$ eliminates quantifiers), and analogously for the
		$\LL^0$-algebraic closures.

		\item The underlying $\LL^0$-structure of a model $M$ of $T$ is
		$\LL^0$-definably closed in some (equivalently, every) model $N$ of
		$T^0$ containing $M$.
		\item In every model $M$ of $T$, the dimension given by the algebraic
		closure with respect to the language $\LL$ coincides with the
		$\LL^0$-Morley rank, computed inside the $\LL^0$-algebraic closure,
		which is strongly minimal.
	\end{itemize}
\end{definition}
		\begin{example}
		Real and $p$-adically closed fields in the ring language 
		are stably controlled by the
		strongly minimal theory ACF$_0$ \cite{HP94}.
	\end{example}

	\begin{lemma}\label{Lem:StablyTame} If a  geometric  
		theory of topological rings $T$ is stably controlled, then it is tame for pairs, that is,   
		in a dense pair $(M,E)$ of models of $T$, every
		$\LL^P$-definably closed set is special. 
	\end{lemma}
In particular, the intersection of $\LL$-definably closed special subsets
	of $M$ is again special, by Fact $\ref{F:Dries}(\ref{F:dcl_special})$.
	\begin{proof} 

			Given an $\LL^P$-definably closed subset $A$ of $M$, we need to
			show that $A\ind_{A\cap E} E$. Choose a model $K$ of the strongly minimal 
			$\LL^0$-theory $T^0$ controlling $T$ such that $M$ embeds into $K$ as an 
			$\LL^0$-substructure. Denote by $\dcl^0$ and $\acl^0$ the 
			definable and algebraic closures in $K$. 
			
			Given a tuple $a$ in $A$, elimination of imaginaries yields that the $\LL^0$-type 
			$q$ of 
			$a$ over $\acl^0(E)$ is stationary. Its canonical base 
			$\text{cb}(q)$ lies in $\dcl^0(E,a) \subseteq M$. Thus 
			$\text{cb}(q)$ 
			lies in $\acl^0(E)\cap M\subseteq \dcl^{\LL}(E)=E$, for $E$ is an 
			elementary $\LL$-substructure of $M$. Any $\LL^P$-automorphism of 
			$M$ fixing 
			$A$ pointwise must permute $E$ as a set, so $\text{cb}(q)$ lies in 
			$\dcl^P(A)=A$. Hence, the restriction of $q$ to $E\cap A$ is stationary, and 
			\[\dim(a/E)=\text{RM}^0(a/E)=\text{RM}^0(a/E\cap A)=\dim(a/E\cap A),\]
			as desired. 
	\end{proof}
	
	If the theory $T$ of $M$ is stably controlled with respect to the strongly minimal theory 
	$T^0$, a verbatim adaptation of
\cite[Theorem A]{HP94} yields that there are:
\begin{itemize}
\item a connected group $(H, \circ)$ which is $\LL^0$-definable in $T^0$ with parameters
from $M$;
\item open $\LL_M$-definable subsets $U$, $V$ and $W$ of
$G$;
\item open $\LL_M$-definable subsets $U'$, $V'$ and $W'$ of the group
$H(M)$;
\item $\LL_M$-definable homeomorphisms $$f:U\rightarrow U'\,,\,
g:V\rightarrow V' \text{ and } h:W\rightarrow W';$$
\end{itemize}
such that
\[ f(x)\circ g(y)=h(x\ast y) \text{ for every } (x,y) \text{
in } U\times V. \]
\noindent In particular, the group law on $G$ is locally $\LL_0$-definable. In the case of
$p$-adically or real closed fields, we conclude that the group law on $G$ is locally
algebraic.

We will prove an analogon for certain small groups definable in dense pairs.

\begin{theorem}\label{T:main} Let $T$ be a geometric NIP theory of 
topological rings, which is stably controlled with respect to the strongly 
minimal $\LL^0$-theory $T^0$. Consider a dense pair $(M,E)$ of models of $T$, with $E=P(M)$, in
the language $\LL^P= \LL\cup \{P\}$.  If $(G, \ast)$ is
  an $\LL^P$-definable group over a special set $D\subseteq M$ such that
  $G\subseteq E^k$ for some $k$, then there are:
\begin{itemize}
\item  an $\LL_E$-definable
  subset $Z$ of $G$ of dimension $\dim(G)$;
\item an $\LL^0$-definable connected  group $(H, \circ)$ over $E$;
\item relatively open $\LL_E$-definable subsets $U$, $V$ and $W$ of
  $Z$ with $U*V\subseteq W$;
\item relatively open $\LL_E$-definable subsets $U'$, $V'$ and $W'$ of
  $H(E)$ with $U'\circ V'\subseteq W'$;
\item $\LL_E$-definable homeomorphisms $$f:U\rightarrow U'\,,\,
  g:V\rightarrow V' \text{ and } h:W\rightarrow W';$$
\end{itemize}
such that
	\[ f(x)\circ g(y)=h(x\ast y) \text{ for every } (x,y) \text{
            in } U\times V. \]
\end{theorem}

\begin{proof}

  The proof is an adaptation of the analogous result \cite[Theorem
  A]{HP94}, so we will only remark the relevant steps, to avoid
  repetitions.  We may assume that the special set $D$ is
  finite-dimensional and $\LL^P$-definably closed, by Fact
  $\ref{F:Dries}(\ref{F:dcl_special})$.

  Let
	\[Y=\{(a,b,a\ast b):a,b\in G\} \subseteq G^3, \]
	\noindent be the graph of the group operation $\ast$. The set $Y$ is
	$\LL^P$-definable over $D$, so choose an honest definition $Z_Y$
	of $Y$ defined over a finite set $A_0$ of $P(N)$, where $N$
	is a sufficiently saturated $\LL^P$-elementary extension $N$
	of $M$, by Fact \ref{F:honest}. Observe that $A_0$ is special in $N$.

	Fact \ref{F:honest} yields that the image $Z_i$ of the
        $i^\text{th}$-projection $\pi_i:Z_Y \rightarrow G(P(N)^k)$ is
        an honest definition of $G$, so
	\[ \dim(Z_1)=\dim(Z_2)=\dim(Z_3)=\dim(G),\]

        \noindent by Corollary \ref{L:good_dim}. Since
        $Z_Y \subseteq \{(a,b,a\ast b):a,b\in G(P(N)^k)\}$, we conclude
        that for each pair $(a,b)$ in
        $$Z_{1,2}:=(\pi_1\times \pi_2)(Z_Y)\subseteq G(P(N)^k)\times
        G(P(N)^k),$$ there is a unique $c$ in $Z_3(P(N)^k)$ such that
        $(a,b,c)$ belongs to $Z_Y$. Hence, the set $Z_Y(P(N)^{3k})$ is
        the graph of an $\LL$-definable function $\odot$ over $A_0$:

\[\begin{array}{cccc} Z_{1,2}(P(N)^{2k})   & \to & Z_3(P(N)^k) \\[1mm]
 (a,b)& \mapsto &a\odot b.
  \end{array}  \]
Note that $Z_{1,2}(P(N)^{2k}) $ is an honest definition of $G\times G$, so
$\dim(Z_{1,2})=2\dim(G)$.

Since $M\preceq^P N$, the set $M$ is special in $N$. In particular, the
set $D$ remains special in $N$, and so is $D_1=\dcl(D,A_0)$, by the Remark
\ref{R:special}. Clearly $D_1$ is finite-dimensional. Observe that  $Z_Y$,
the projections $Z_1$, $Z_2$ and $Z_3$, as well as the map $\odot$ are
all $\LL$-definable over $D_1$.

By Lemma \ref{L:dim_hon}, choose a generic point $(a,b)$ in
$Z_{1,2}(P(N)^{2k})$ with $c=a\odot b$ in $Z_3$, so
 $$\dim(a,b/D_1)=2\dim(G).$$
 \noindent The pair $(a,b)$ is the starting point of the group
 configuration theorem, working in the $\LL^0$-algebraic closure $\acl^0(P(N))$ with 
 respect to the theory $T^0$.  

 Given any two elements $x$ and $y$ of $G(P(N)^k)$, the set
 $D_1\cup\{x,y\}$ is again special in $N$, by the Remark \ref{R:special}.
 Let $P(D_1)$ denote $D_1\cap P(N)$. By
 Fact $\ref{F:Dries}(\ref{F:dcl_special})$, we have that
 $P(\dcl^P(D_1\cup\{x,y\})= \dcl(P({D_1}), x,y)$, so the element
 $x\ast y$ lies in $\dcl(P(D_1), x,y)$. Moreover, given any special set
 $D'\supseteq D$, if $x$ and $y$ are both
 generic in $G(P(N)^k)$ and independent over $D'$ then so is $x\ast y$, which is independent from each
 of the factors over $D'$.  
 
 Because of this, the same arguments as in \cite[Proposition
 3.1]{HP94} carry on to obtain a finite-dimensional subset $D_2$ of
 $P(N)$ containing $P(D_1)$, a connected group
 $H\subseteq \acl^0(P(N))^\ell$ which is $\LL^0$-definable over $D_2$ in the theory 
 $T^0$, and
 points $a'$, $b'$ and
 $c'$ of $H(P(N)^\ell)$, such that

\begin{itemize}
\item  $a'$ and $b'$ are $D_2$-generic independent elements of
$H(P(N)^\ell)$;
 \item $a' \cdot b' = c'$;
\item $\dcl(aD_2) = \dcl(a'D_2)$, $\dcl(bD_2) = \dcl(b'D_2)$ and
$\dcl(cD_2) =
\dcl(c'D_2)$.
 \end{itemize}

 \noindent In particular, there is an $\LL_{D_2}$-definable bijection
 $f$ from a subset of $G(P(N)^k)$ to a subset of $H(P(N)^\ell)$, with
 $f(a)=a'$.  We now proceed as in \cite[Lemma 4.8]{HP94}. Denote by
 $Z$ the $\LL$-definable set $Z_1\cup Z_2\cup Z_3$. Note that
 $Z(P(N)^k)\subseteq G(P(N)^k)$, and all three elements $a$, $b$ and
 $c=a\ast b$ lie in $Z$. Since $a$ in $Z$ and $a'$ in $H(P(N)^\ell)$
 are each generic over $D_2$, there are relatively open $\LL_{D_2}$-definable
 neighbourhoods $U$ of $a$ in $Z$ and $U'$ of $a'$ in $H(P(N))$, such
 that $f:U\rightarrow U'$ is a homeomorphism.  Similarly, there exists
 $\LL_{D_2}$-definable homeomorphisms $g:V\rightarrow V'$ and
 $h:W\rightarrow W'$ with $g(b)=b'$ and $h(c)=c'$, for some relatively open
 $\LL_{D_2}$-definable open subsets $V$ and $W$ of $Z$, resp. $V'$ and
 $W'$ of $H(P(N))$.

  Now, the function $\odot$ is $\LL_{A_0}$-definable and
 $A_0\subseteq P(D_1)\subseteq D_2$. The
 $\LL_{D_2}$-definable set
$$\{(x,y)\in Z_{1,2}\cap (U\times V)\cap \odot^{-1}(W): f(x)\circ
g(y)=h(x \odot y)\} \subseteq Z\times Z,$$
\noindent  contains the point $(a,b)$, which has dimension $2\dim(Z)$ over
$D_2$. There are open
$\LL_{P(N)}$-definable neighbourhoods $U_1$ of $a$, resp. $V_1$ of $b$,
in $Z$ such that $U_1\times V_1$ is contained in the above set.
Hence, we have that
$$U_1\subseteq U\,,\, V_1\subseteq V \text{ and } U_1 * V_1=U_1\odot
V_1\subseteq W$$
\noindent with
$f(x)\circ g(y)=h(x\odot y)=h(x*y)$ for every $(x,y)$ in $U_1\times
V_1$.
Finally, rename $U_1$ and $V_1$ by $U$ and
$V$, respectively. Since $M\preceq^P N$, we can obtain the
corresponding definable sets and homeomorphisms in our pair $(M,E)$,
all definable over $E$, with $Z\subseteq G$ of dimension $\dim(G)$.
\end{proof}

If the theory $T$ is an o-minimal expansion of a real closed field, we can provide an 
intrinsic definition of the local isomorphism of Theorem \ref{T:main}. 
	
Recall that (the theory 
of) a weakly minimal expansion of an ordered abelian group $M$ 
is 
\emph{non-valuational} if there are no 
proper definable subgroups of $M$ (cf. \cite[Lemma 1.5]{W08}). Every o-minimal theory is 
non-valuational, since they are definably complete, so definable cuts are realized. A typical 
example of a valuational theory is the theory of a 
non-archimedian real closed field 
with a distinguished predicate for the convex hull of the integers. 

Let $(M,E)$ now be a dense pair of models of $T$ and fix a special subset $D\subseteq 
M$. The (weak) \emph{Shelah's expansion} \cite{sS14} $E_\text{ext}$ of $E$ is the 
structure  in the language $\LL^{\text{ext}}$ consisting of relational symbols $R_X$ 
for each $\LL_D^P$-definable subset $X$ of $E^n$, as $n$ varies,  whose universe is $E$
 and the interpretation in $E_\text{ext}$ of each $R_X$ is $X$.  Since the projection of an 
 $\LL_D^P$-definable subset  of $E^n$ is again $\LL_D^P$-definable, the theory 
 $\text{Th}(E_{\text{ext}})$ eliminates quantifiers. Thus, definable sets in  $E_{\text{ext}}$ 
 are  exactly the $\LL_D^P$-definable sets, which are exactly the 
 traces on $E$ of  $\LL_D$-definable subsets of $M$, by Fact $\ref{F:Dries} 
 (\ref{F:induced})$. In particular, the structure $E_\text{ext}$ is a weakly o-minimal 
 \cite{BP98}. Furthermore, it is non-valuational: given a subgroup $H$ of $(E,+, \leq)$ 
 definable in $E_{\text{ext}}$, the set $\text{cl}_M(H)$ is a subgroup of $(M,+)$ and it is 
 $\LL$-definable, by Theorem \ref{T:opendef}. Since $M$ is o-minimal,  it follows that 
 $\text{cl}_M(H)$ is trivial or $\text{cl}_M(H)=M$. As the characteristic is $0$, the only 
 finite additive subgroup is the trivial one. Assume hence that $\text{cl}_M(H)=M$, so 
 $\dim(E\setminus H) < \dim(H)=1$, by Proposition 
 $\ref{P:dim_tame}$. In particular, the set $E\setminus H$ 
 has dimension $0$, so it is finite, which yields that $E=H$, since 
 $E$ is divisible. 
 
 Hence, the (honest) dimension defined in Definition \ref{D:dim_hon} for the 
 non-valuational weakly o-minimal structure $E_{\text{ext}}$ is a dimension in the sense 
 of  \cite{W12}, for it coincides with the  topological dimension, by Lemma 
 \ref{L:proj_dim_coordinates}. Theorem \cite[Theorem 4.6]{W12} yields that every 
 $\LL_D^P$-definable  group $G\subseteq E^k$ (which are $\LL_{\text{ext}}$-definable) 
 has a topology $\tau$ which turns it into a topological group. Specifically, there is an 
 $\LL_D^P$-definable $\tau$-open subset $L$ of $G$ such that the topology $\tau$ 
 restricted to $L$  coincides with the topology induced from $E^k$ and  finitely many 
 translates of $L$  cover $G$, with $\dim(G\setminus L)<\dim(G)$. 
 
 We can now conclude the following result:
 
 \begin{cor}\label{C:Wenzel} Let $T$ be an o-minimal expansion of a real closed field, 
 which is stably 
controlled by the strongly minimal $\LL^0$-theory $T^0$. Consider a dense
pair $(M,E)$ of models of $T$, with $E=P(M)$, in
the language $\LL^P= \LL\cup \{P\}$.  If $(G, \ast)$ is
  an $\LL^P$-definable group over a special set $D\subseteq M$ such that
  $G\subseteq E^k$ for some $k$, then there is an 
  $\LL^P$-definable local isomorphism 
  between a neighbourhood of 
  the identity in $G$ and a neighbourhood of the identity in 
  $H(E)$, for some connected group 
  $H$ which is $\LL^0$-definable over $E$. 
\end{cor}  

\begin{proof} 
	Enlarging $D$, we may assume that all the sets and maps in Theorem \ref{T:main} are 
	definable over the finite-dimensional special set $D$. Let now $\tau$ be the topology of 
	$G$ described above and $L\subseteq G$ be the $\LL^P_D$-definable $\tau$-open  
	subset of $G$ such that $\tau$ restricted to $L$ coincides with the topology induced 
	from $E^k$ and  finitely many 
	translates of $L$  cover $G$, with $\dim(G\setminus L)<\dim(G)$. 
	
 We first show that the subset $U\cap L$ has non-empty interior in the set $U$ obtained 
 in Theorem  \ref{T:main}. Since  \[ U=(U\cap L) \cup (U\setminus L),\]
 we have that $\dim(U\cap L)=\dim(U)=\dim(G)$, for $\dim(G\setminus L)<\dim(G)$. 
 By Proposition $\ref{P:dim_tame} (\ref{P:dim_topcl})$, the dimension of the set 
 $\cl_{M^n}(U\setminus L)$ equals $\dim(U\setminus L)<\dim(G)$. Now, 
 
 \[  (U\cap L)\setminus \text{int}_U(U\cap L) \subseteq \cl_{M^n}(U\setminus 
 L),  \]
 so $\dim(\text{int}_U(U\cap L))=\dim(U\cap L)$, hence $\text{int}_U(U\cap L)\neq 
 \emptyset$, as desired. Observe that  $\text{int}_U(U\cap L)$ remains relatively open in 
 $Z$, since $U$ was.  After renaming, we may hence assume that $U$ is 
 $\LL_D^P$-definable and open in $L$.  Similarly, so are the sets $V$ and $W$, after 
 renaming. 
 
 Note that the topological dimension in  the non-valuational weakly o-minimal structure 
 $E_{\text{ext}}$ agrees with the dimension determined by the pregeometry induced by 
 the definable closure $\text{dcl}^P(-,D)$ on $E$, by  
 \cite[Theorem 4.6]{W12}.  
 
 Pick a pair $(a,b)$ in $U\times V$  such that $\dim^P(a,b|D)=2\dim(G)$, and set 
 $c=a\ast b$. Since $G$ is a topological group, we may assume that $U^{-1}*c\subseteq 
 V$, so  $U*U^{-1}*c\subset W$.  Set $U_1$ the  $\LL^P_D$-definable $\tau$-open 
 neighbourhood $U^{-1}\ast a$ of the identity of $G$, and $U'_1$ the $\LL_E$-definable 
 open neighbourhood $(U')^{-1} \circ a'$ of the identity of $H(E^k)$. As shown in 
 \cite[Corollary 4.9]{HP94}, the $\LL^P_D$-definable map
\[\begin{array}{ccc}
U_1  & \to & U'_1 \\[1mm]
 x^{-1}\ast a &\mapsto & f(x)^{-1}\circ a'
\end{array} \]
 is a homeomorphism and a local isomorphism, as required.
\end{proof}

\section{Digression: Elimination of imaginaries for the predicate}\label{S:EI}
Eleftheriou showed \cite{pE17} that, given any expansion of an
o-minimal ordered group by a new predicate, the induced structure on
the predicate has elimination of imaginaries whenever it satisfies
three natural (yet technical) conditions, which hold for all canonical
examples of pairs of an o-minimal expansion of an ordered group by a
predicate. His result allows to characterize small sets and it will  
be crucial in order to describe arbitrary 
small groups.  For the sake of the 
presentation, we will provide a short proof of Eleftheriou's result, which 
becomes somewhat easier in our particular setting.

Recall that the language $\LL$ expands the language of rings and that the
ring operations are continuous with respect to the definable topology. 
Furthermore, we require that:
\begin{itemize}
\item the geometric NIP theory $T$ eliminates 
imaginaries;
\item the intersection of two definably closed special sets is 
special: this is the case, whenever $T$ is stably controlled with 
respect to a strongly 
minimal theory $T^0$, by Lemma \ref{Lem:StablyTame}.
\end{itemize}

\begin{cor}\label{C:intersection}\cite[Lemma 3.1]{pE17}

	Let $A$ and $A'$ be two definably closed special subsets of $M$.
	Given a subset $Y$ of  $E^n$, which is both $\LL^P$-definable over
	both $A$ and $A'$, the set $Y$ is $\LL^P$-definable over $A\cap A'$.
\end{cor}

\begin{proof}
	The proof goes by induction on $\dim(Y)$. For $\dim(Y)=0$, it is
	obvious, since $Y$ is finite. Otherwise, if $\cl_{M^n}(Y)$ denotes the 
	topological closure of $Y$, note that
	$\cl_{M^n}(Y)\cap E^n$ is the disjoint union of $Y$ and the set
	$$(\cl_{M^n}(Y)\cap E^n)\setminus Y=(\cl_{M^n}(Y)\setminus Y) \cap
	E^n.$$
	\noindent Therefore, it suffices to show that both
	$\cl_{M^n}(Y)\cap E^n$ and $(\cl_{M^n}(Y)\cap E^n)\setminus Y$ are
	$\LL^P$-definable over $A\cap A'$. By Theorem \ref{T:opendef}, the
	set $\cl_{M^n}(Y)$ is $\LL$-definable over $A$ and $A'$, so it is
	$\LL$-definable over $A\cap A'$, by elimination of imaginaries of
	the theory $T$.

	 Finally, the dimension of the $\LL^P$-definable set
	$(\cl_{M^n}(Y)\setminus Y)\cap E^n$ is strictly less than $\dim(Y)$,
	by Proposition $\ref{P:dim_tame} (\ref{P:dim_bound})$. Since
	$(\cl_{M^n}(Y)\setminus Y)\cap E^n$ is both $\LL^P$-definable over
	$A$ and $A'$, it is definable over $A\cap A'$, by induction.
\end{proof}

 We say that a definable set $Z$ in a fixed complete theory
has a \emph{field of definition} if there is a smallest definably
closed subset $B$ over which $Z$ is definable. Clearly, a theory with
elimination of imaginaries has fields of definitions for all definable
sets.  If a theory has fields of definitions for all definable sets,
then it has weak elimination of imaginaries. Indeed, given an
imaginary $e=a/E$, let $B$ be a field of definition of the definable
class $E(x,a)$. By minimality, the set $B$ equals $\dcl(b)$, for some
finite tuple $b$ such that $e$ is definable over $b$. We need only
show that the type $\tp(b/e)$ is algebraic.  Observe that the
collection of realisations of $\tp(b/e)$ lies in the set $B$, so its
cardinality is bounded. Compactness yields that this set must be
finite, so $b$ is algebraic over $e$, as desired.

\begin{fact}\cite[Theorem C and Corollary 1.4]{pE17}\label{L:Pantelis}

	For every $\LL^P$-definable small subset $X \subseteq M^n$ over a
	special set $A$, there is an $\LL^P$-definable injection
	$X$ into $E^l$ over $A$, for some natural number $l$.
\end{fact}

\begin{proof}
	By Fact $\ref{F:Dries}(\ref{F:dcl_special})$, we may assume that $A$
	is $\LL^P$-definably closed.  By smallness of $X$ over $A$, there is
	an $\LL_A$-definable function $h:M^m\to M^n$, with
	$X\subseteq h(E^m)$. For tuples of $E^m$, set
	$$ e_1\sim e_2 \text{ if } h(e_1)=h(e_2).$$

	Let us first show that, for each tuple $e$ in $E^m$, its $\sim$-class
	has a field of definition in $T^P$, working over the parameter set
	$A$.  Clearly, the $\sim$-class $Y$ of $e$ is $\LL^P$-definable over
	$\dcl(Ae)$, since the latter is special and $\LL^P$-definably closed,
	by the Remark \ref{R:special} and Fact
	$\ref{F:Dries}(\ref{F:dcl_special})$.  Choose a tuple $d$ in $E$ of
	minimal length such that the set $Y$ is definable over the special set
	$\dcl(Ad)$. Suppose that $\dcl(Ad)$ is not a field of definition of
	$Y$. Hence, there is an $\LL^P$-definably closed set $B\supseteq A$
	(hence special) over which $Y$ is defined, but
	$$ A\subseteq A'=\dcl(Ad)\cap B\subsetneq \dcl(Ad).$$

 Our assumptions yield that the set $A'$ is special, since both $B$ and 
 $\dcl(Ad)$ are. By
	Corollary \ref{C:intersection}, the set $Y$ is definable over $A'$.
	Choose an $A$-independent tuple $u$ in $M$ with $A'=\dcl(Au)$. Though
	$u$ need not be in $E$, it lies in $\dcl(Ad)\subseteq \dcl(AE)$. The
	independence

	$$A,u \ind_{E\cap \dcl(Au)} E,$$

	\noindent implies that $\dim(u/A, E\cap \dcl(Au))=0$. Choose a tuple
	$f$ of $E$ of length $|u|$ which is $\LL$-interdefinable with $u$ over $A$.
	By the exchange principle, we have that $|f|<|d|$,
	contradicting our choice of $d$. Thus, the set $Y$ has $\dcl(Ad)$ as a
	field of definition (working over $A$).

 Since $T$ has finite Skolem functions, the set $Y$ has a canonical parameter in $E$, working over $A$. By compactness,
	there exists an $\LL^P_A$-definable injection which sends the
	$\sim$-class of $e$ to some tuple in $E^l$, for some fixed $l$. Every
	point of $X$ is of the form $h(e)$, for some $e$ in $E^m$. Any two
	representatives lie in the same $\sim$-class, so composing we obtain
	an $\LL^P$-definable injection $X\hookrightarrow E^l$ over $A$, as
	desired.

\end{proof}

Fact \ref{L:Pantelis} yields that every $\LL^P$-definable small
group of $(M,E)$ is, up to $\LL^P$-definability, a subset of some 
cartesian product $E^k$,  so
we conclude the following:

\begin{cor}\label{C:small_gps}
If the tame topological geometric theory $T$ with NIP eliminates 
imaginaries, is stably controlled by a theory $T^0$ and expands
a topological field,  then every small group $(G,\ast)$ has
locally, up to $\LL^P$-interdefinability, an $\LL^0$-definable 
group law with parameters
from $E$.
\end{cor}

\begin{question}
Is the group law of every small group in a dense pair of real closed fields  semialgebraic,
up to $\LL^P$-interdefinability?
\end{question}

\section{Appendix:  reproving Fact \ref{F:Dries}}\label{A:Dries}

In \cite[Corollary 3.4]{vD98}, it was stated that $\LL^P$-definable 
function  $F:M\rightarrow M$  agrees off some small subset of 
$M$ with a function
$\widehat{F}:M\rightarrow M$ that is $\LL$-definable. However, 
the proof only yields that $F$ agrees off some small subset of 
$M$ with one of finitely many $\LL$-definable functions 
$\widehat{F}_1,\ldots,\widehat{F}_\ell:M\rightarrow M$. Note that 
the same gap also affects \cite[Theorem 4.9]{BDO11}. 

After private communication with van den Dries, we will provide in 
this appendix a proof of \cite[Corollary 3.4]{vD98} in the broader 
context of geometric topological structures, fixing the gap in the 
published version. 

\medskip
{\bf In this appendix, we work in a sufficiently saturated dense 
pair of models $(M,E)$ of a geometric theory $T$ of topological 
rings}.
\medskip

As in \cite[Theorem 2.5]{vD98}, a straightforward adaptation of 
the back-and-forth system
yields items (\ref{F:complete}), (\ref{F:special}) and 
(\ref{F:induced}) in Fact \ref{F:Dries}. 
	
\begin{lemma}\label{L:familyopen} Let $D$ be a special subset 
of $M$ and $(U_y:y\in Y)$ a uniformly $\LL_D$-definable 
family of open subsets of $M$, parametrized over the 
$\LL_D$-definable subset $Y\subseteq M^k$. The set
		$$\bigcup_{y\in Y\cap E^k}U_y$$
		is $\LL_D$-definable. 
	\end{lemma}

	\begin{proof} Naming the parameters in $D$, we may clearly 
	assume that $D=\emptyset$. By cell decomposition, we may 
	assume that $Y$ is a cell. We will prove it by induction on 
	$k$, 	the initial case $k=0$ being trivial. 
		
		\begin{claim} We may assume that $\dim(Y)=k$.
				\end{claim}
			\begin{claimproof}
Otherwise,  there is some projection $\pi: M^k\rightarrow 
M^\ell$, with $\ell<k$, such that $\pi\restr Y$ is a 
homeomorphism 
between $Y$ and the open subset $\pi(Y)$ of 
$M^\ell$. By Fact  \ref{F:Dries} (\ref{F:induced}), the 
$\LL^P$-definable subset $\pi(Y\cap E^k)\subset E^\ell$ is the 
trace of an $\LL$-definable set $Y_1$, that is, we have that 
$Y_1\cap E^\ell=\pi(Y\cap E^k)$. Considering the intersection, 
we may assume that $Y_1\subseteq \pi(Y)$, so denote by 
$\rho:Y_1\rightarrow Y$ 
 the inverse of $\pi\restr Y$ restricted to $Y_1$. 
Reparametrizing the family $(U_y:y\in Y)$ as $(U_{\rho(y)} :y\in 
Y_1)$ and decomposing $Y_1$ into cells, we conclude the desired 
result by induction. 
			\end{claimproof}
		
	Hence, the cell $Y$ is open in $M^k$. For $x$ in the open 
	$\LL$-definable 
		subset 
		$$U=\bigcup_{y\in Y} U_y$$ of $M$, consider the 
		corresponding 	$\LL$-definable sets 
		$$Y_{x}=\{y\in Y\,|\, x\in U_y\},$$
		and  $$B=\{x\in U\,|\,  
		\text{int}(Y_{x})=\emptyset \}.$$
		
	\begin{claim}
The set $B$ is finite. 
	\end{claim}	
\begin{claimproof}
	If $B$ were infinite, its interior $\text{int}(B)$  is non-empty. 
Set \[ Y_1=\bigcup\limits_{x\in \text{int}(B)} Y_x= 
\{y \in Y\,| U_y\cap 
\text{int}(B)\neq \emptyset\}.\]

 For each $y$ in $Y_1$, the open set 
$U_y\cap \text{ int}(B)$ is not empty. By Skolem functions in the 
theory $T$, choose some $x=f(y)$ in 
$U_y\cap \text{int}(B)$, for some $\LL$-definable map $f:Y_1\to M$. In 
particular, there is a basic open neighborhood of $f(y)$ 
contained in $U_y\cap \text{int}(B)$, that is, 
\[ f(y) \in \theta(M,g(y))\subseteq  U_y\cap \text{int}(B),\]
for some $\LL$-definable Skolem function $g:Y_1\to M^s$. We 
may assume, partitioning $Y_1$, that both $f$ and $g$ are 
continuous. 

If $Y_1$ were not empty, choose some element $y$ in $Y_1$ and 
consider $g(y)$ in $M^s$. By the density condition 
\ref{D:geomstr},  there is an open neighborhood $V$ of $g(y)$ and a tuple $\bar 
c$  such that for all $\bar d$ in $V$, 
\[ f(y) \in \theta(M,\bar c) \subseteq \theta(M,\bar d).\]

Let us first show that the open neighborhood $g\inv(V)$ of $y$ 
is contained in $Y_{f(y)}$: given $z$ in 
$g\inv(V)$, we need to show that $f(y)$ belongs to $U_z$. Now, 
the element $g(z)$ belongs to $V$, so 

\[ f(y) \in \theta(M,\bar c) \subseteq \theta(M,g(z)).\]
\noindent By construction, the set $\theta(M,g(z)) 
\subseteq U_z\cap \text{int}(B)$, so $f(y)$ lies in $U_z$ as desired. Hence 
$Y_{f(y)}$ has non-empty interior, which contradicts that $f(y)$ belongs to 
$B$. In particular, the set $Y_1$ 
is empty and thus so is $\text{int}(B)$, since 
$Y_x\neq\emptyset$, for every $x$ in $U$.  Therefore, 
the set $B$ is finite, as desired. 
	\end{claimproof}
	
	Now, the finite subset
\[B'=B\cap \big(\bigcup_{y\in Y\cap 
	E^k}U_y\big)\] is trivially $\LL$-definable. Note 
	that  
\[ \bigcup_{y\in Y\cap E^k}U_y=B'\cup \bigcup\limits_{y\in 
	Y\cap E^k}(U_y\setminus B).\] We need only show that  
	\[\bigcup\limits_{y\in 
	Y\cap E^k}(U_y\setminus B) = \bigcup\limits_{y\in 
	Y} (U_y\setminus B),\] in order to conclude that $ \bigcup_{y\in 
	Y\cap E^k} U_y$ is $\LL$-definable. 

Given $x$ in $(U_y\setminus B)$ for 
some $y$ in $Y$, we have that $\text{int}(Y_{x})\neq 
\emptyset$, since $x$ does not lie in $B$. Hence, there is an 
open neighbourhood $V$ of $y$ contained in $Y_x$, so choose 
an element $y_0$ in $V\cap E^k$, by density. Thus $x$ lies in 
$U_{y_0}\subseteq \bigcup_{y\in Y\cap E^k}  U_y$, as desired. 
\end{proof}

\begin{fact}\label{F:quantifier}(\cite[Theorem 2.5]{vD98}) Every 
$\LL^P$-formula $\phi(x_1,\ldots,x_n)$ is equivalent in $T^P$ to 
a boolean combination of formulae of the form 
$$\exists y_1 \cdots \exists y_m \big(P(y_1)\wedge \cdots 
\wedge P(y_m) \wedge  \psi(x_1,\ldots,x_n, 
y_1,\ldots,y_m)\big)$$
where for some $\LL$-formula $\psi(x_1,\ldots,x_n, 
y_1,\ldots,y_m)$. 
\end{fact}

We will now describe $\LL^P$-definable subsets of $M$, which is 
exactly \cite[Corollary 3.5]{vD98}.  We will obtain such a 
description using Fact \ref{F:quantifier} (instead of deducing it 
from \cite[Cor 3.4]{vD98}).

\begin{prop}\label{P:descrip1ario} For every $\LL^P$-definable 
subset $S\subseteq M$ over an special set $D$, there are 
$\LL_D^P$-definable small sets  $X_1$ and $X_2$, and an  
$\LL_D$-definable subset $S'\subseteq M$ such that
$$S=(S'\setminus X_1) \cup X_2.$$
\end{prop}
\begin{proof} Let $\mathcal{F}$ denote the family of subsets of 
$M$ of the form 
$$(S'\setminus X_1)\cup X_2$$
with $S'$, $X_1$ and $X_2$ as in the statement.  Since every 
subset 
of a 
small subset is again small,  it is easy to check that $\mathcal{F}$ 
is closed under boolean combinations. 
Thus, by Fact \ref{F:quantifier}, we need only show that each unary 
definable set $S$ of the form 
$$\exists y_1 \cdots \exists y_m \big(P(y_1)\wedge \cdots 
\wedge 
P(y_m) \wedge  \psi(x,d,y_1,\ldots, y_m)\big)$$
belongs to $\mathcal{F}$, where $d$ is a tuple of parameters in 
$D$. Set $Y\subseteq M^m$ be the set $\exists x 
\psi(x,d,y_1,\ldots, y_m)$. By uniform cell decomposition, for 
each $y$ in $Y$, the $\LL_D$-definable subset of $\psi(M,d,y)$ of 
$M$ is a union of a finite set $F_y$ (whose cardinality is bounded 
by some $r$ in $\N$ uniformly on $y$) and an
open $\LL_D$-definable  subset $U_y$ of $M$. Since $T$ 
 has finite Skolem functions, there are 
$\LL_D$-definable functions $g_1,\ldots,g_r:Y \rightarrow M$ 
such that $F_y=\{g_1(y),\ldots,g_k(y)\}$ for each $y$ in $Y$. 

 The  $\LL_P$-definable set \[X_2=\bigcup\limits_{y\in Y\cap 
 E^m} F_y=\bigcup_{i=1}^r g_i(E^k),\] is small. By Lemma 
 \ref{L:familyopen}, the set \[S'= \bigcup\limits_{y\in Y\cap E^m}  
 U_y\] is $\LL_D$-definable. Clearly, the set $S=S'\cup X_2$, as 
 desired.
\end{proof}
\begin{remark} In particular, if a small subset $X\subseteq M$ is 
$\LL^P$-definable over some set of parameters $D$, then there is 
an $\LL_D$-definable map $h:M^m\rightarrow M$ such that 
$X\subseteq h(E^m)$.
\end{remark}

The proof of \cite[Theorem 4]{vD98}, and more generally of 
\cite[Theorem 4.9]{BDO11}, only need as main ingredient the  
Proposition 
\ref{P:descrip1ario}, and not  \cite[Corollary 3.4]{vD98}. Hence, we 
obtain verbatim a proof of the following result: 

\begin{prop}\label{P:onedimensional} 
(cf. \cite[Theorem 4]{vD98}) For every $\LL^P$-definable unary set
  $X\subseteq M$ over a special set $D$, there are pairwise 
  disjoint $\LL_D$-definable open
  sets $U$, $V$, $W_1$ and $W_2$ of $M$  such that:
  \begin{itemize}
  \item $M\setminus (U\cup V\cup W_1 \cup W_2)$ is finite.
  \item $U\subseteq X$ and $X\cap V=\emptyset$.
  \item $X\cap W_1$ is dense in $W_1$ and small.
  \item $X\cap W_2$ is dense and codense in $W_2$ and co-small in $W_2$.
  \end{itemize} 
\end{prop}

We have now all the ingredients in order to fill the gap of 
\cite[Corollary 3.4]{vD98}.

\begin{theorem}\label{T:Function_fixed} Any $\LL^P$-definable 
function  $F:M\rightarrow M$ definable over an special set $D$ 
agrees off some small subset of $M$ with an $\LL_D$-definable  
function $\widehat{F}:M\rightarrow M$. 
\end{theorem}
\begin{proof} As stated in the beginning of  the Appendix, the 
proof of \cite[Corollary 3.4]{vD98}  yields that there are finitely 
many $\LL_D$-definable maps 
$\widehat{G}_1,\ldots,\widehat{G}_\ell:M\rightarrow M$ and a 
small $\LL^P$-definable set $X$ such that 
for all $x$ in $M\setminus X$, we have that
$F(x)=\widehat{G}_i(x)$ for some $1\leq i\leq\ell$. 

Clearly, if $\ell=1$, we are done. Otherwise, consider the set  
$X_1=\{x\in M \,|\, F(x)= \widehat{G}_1(x)\}$, and let $U$, $V$, 
$W_1$ 
and $W_2$ be the corresponding open $\LL_D$-definable subsets 
of $M$ for $X_1$, as in Proposition \ref{P:onedimensional}. Define 
$F_1=M\setminus (U\cup V\cup W_1 \cup W_2)$ and 
\[ \begin{array}[t]{rccc}
\widehat{G_2'}:&M &\to &M\\
& x& \mapsto &\begin{cases} \widehat{G}_1(x) \text{ if } x\in 
U\cup 
	W_2 \\ \widehat{G}_2(x)  \text{ otherwise.}
\end{cases}\\
\end{array}\]

\noindent The subset $X'=X \cup (X_1\cap 
W_1)\cup (W_2\setminus X_1) \cup F_1$ is clearly small, and 
note that for $x$ in 
$M\setminus X'$, the value $F(x)$ equals $\widehat{G_2'}(x)$
or $\widehat{G}_i(x)$, for some $3\leq i\leq\ell$. We proceed now 
by induction on $\ell$ in order to produce an
$\LL_D$-definable function
$\widehat{F}$, which agrees with $F$ off a small subset of $M$. 
\end{proof}

\end{document}